\documentclass{amsart}

\usepackage{mathptmx}
\usepackage{amsmath,amsthm}
\usepackage{amscd}
\usepackage{amsfonts,enumerate}
\usepackage{amssymb}
\usepackage{graphicx,cite}
\usepackage{mathrsfs}
\usepackage[all,cmtip,2cell]{xy}

\newtheorem{dfn}{Definition}[section]
\newtheorem{thm}{Theorem}[section]
\newtheorem{lem}[thm]{Lemma}
\newtheorem{pro}[thm]{Proposition}
\newtheorem{cor}[thm]{Corollary}
\theoremstyle{remark}
\newtheorem{rmk}{Remark}[section]

\title{Cross-connection structure of concordant semigroups}
\thanks{\noindent The first author acknowledges the financial support of the Competitiveness Enhancement Program of Ural Federal University, Russia during the preparation of this article.}

\author{P. A. Azeef Muhammed }
\address{Institute of Natural Sciences and Mathematics, Ural Federal University, 620000 Ekaterinburg, Russia.}
\email{azeefp@gmail.com, a.a.parail@urfu.ru}
\author{P. G. Romeo} 
\address{Department of Mathematics,
Cochin University of Science And Technology,
Cochin- 682022, India.}
\email{romeopg@cusat.ac.in}
\author{K. S. S. Nambooripad}
\address{Department of Mathematics, University of Kerala, Thiruvananthapuram- 695581, India.}
\email{kss267@gmail.com}


\begin{document}

\maketitle

\begin{abstract}
Cross-connection theory provides the construction of a semigroup from its ideal structure using small categories. A concordant semigroup is an idempotent-connected abundant semigroup whose idempotents generate a regular subsemigroup. We characterize the categories arising from the generalised Green relations in the concordant semigroup as consistent categories and describe their interrelationship using cross-connections. Conversely, given a pair of cross-connected consistent categories, we build a concordant semigroup. We use this correspondence to prove a category equivalence between the category of concordant semigroups and the category of cross-connected consistent categories. In the process, we illustrate how our construction is a generalisation of the cross-connection analysis of regular semigroups. We also identify the inductive cancellative category associated with a pair of cross-connected consistent categories.
\end{abstract}

\section{Background and overview}\label{sec1}

Semigroups are natural, yet rather general algebraic objects. Hence structure theorems of semigroups are quite elusive and often provided using partially ordered sets, semilattices, groups, groupoids, small categories etc. as the basic building blocks. Due to the inherent generality of arbitrary semigroups, the historical development of structure theory dealt with special classes of semigroups which admit a simpler structure. One of the first structural results in semigroup theory known as the Rees-Sushkevich theorem described the structure of completely $0$-simple semigroups using groups and sets \cite{rees}. 

Later, the search of a breakthrough in more general classes of semigroups shifted the focus onto fundamental semigroups. A semigroup is said to be \emph{fundamental} if it cannot be shrunk homomorphically without collapsing its idempotents; roughly speaking, a fundamental image of a semigroup will provide its `structural skeleton'. In 1970, Munn \cite{munn} described the structure of fundamental inverse semigroups by exploring the semilattice of idempotents of the semigroup. Inspired by Munn's construction, two  approaches to study the structure of regular semigroups were established in the early seventies. The first approach initiated by the third author \cite{ksssf,mem} involved the characterisation of the structure of the idempotents of a regular semigroup as a regular biordered set and building a fundamental regular semigroup as an exact generalisation of Munn's construction. In \cite{mem},  the construction was extended to arbitrary regular semigroups using inductive groupoids (specialised small groupoids whose identities form a regular biordered set) and, a category equivalence between the category of regular semigroups and the category of the inductive groupoids was also proved. This work was a major milestone in the context of the so-called Ehresmann-Schein-Nambooripad (ESN) theorem \cite{lawson} and its later generalisations in several directions \cite{lawson,hollings,gomes,gould2,gouldrestr}. 

The second approach initiated independently by Hall \cite{hall} relied on the idea that instead of encoding the idempotent structure of a regular semigroup as a set equipped with two orders, one can explore the ideal structure of the semigroup and use two partially ordered sets (one each from the principal left and right ideals) to build the fundamental semigroup. Grillet \cite{gril,gril1,gril2} clarified Hall's ideas by characterising such partially ordered sets as regular partially ordered sets (regular posets). Grillet also introduced the notion of cross-connections to explicitly describe the relationship that must exist between the regular posets so that they give rise to a fundamental regular semigroup. 

In \cite{bicxn}, the third author proved the equivalence of the two constructions in the fundamental case. Later, elaborating on that equivalence, he \cite{cross} successfully extended Grillet's cross-connection construction to arbitrary regular semigroups (not only fundamental ones) by replacing regular posets with what he called normal categories. These normal categories are essentially small categories whose identities form a regular poset. In \cite{cross}, a category equivalence between the category of regular semigroups and the category of cross-connected normal categories was also proved. Recently, the first author and Volkov \cite{indcxn1,indcxn2} showed the direct equivalence of the above discussed approaches to arbitrary regular semigroups: the ESN approach and the cross-connection approach.

Recall that the structure of an arbitrary semigroup is naturally composed of two components: left and right. The very fact that we need a `bi'-ordered set (each quasi-order coming from each of the left and right Green relations), to characterise the idempotents of a semigroup, is a reflection of this property. Hence, except for special classes of semigroups (such as inverse semigroups) with additional `structural symmetry', the natural approach to a structural description should use two categories (and a cross-connection to connect them) and not one (as in the ESN approach).

Several classes of semigroups (beyond inverse and regular) were studied using `ESN like' approaches in the last three decades \cite{lawson,hollings,gomes,gould2,gouldrestr,wangureg,swang2017,swang2018,stokes2017}. But this literature is almost reaching the saturation point, due to the obvious constraint of relying on a suitable set of idempotents of the semigroup. One major advantage which the cross-connection approach has over the ESN approach is that is not restricted to the `idempotent part' of the semigroup as it relies on the ideal structure of the semigroup. Hence, we can study more general classes of semigroups using cross-connections. In this article, we use the cross-connection theory to study a class of non-regular semigroups called 
concordant semigroups; thus, the first success in the cross-connection theory, beyond the regular case. 

Concordant semigroups were introduced and studied by Armstrong \cite{armstrong}, as generalisations of regular semigroups. Concordant semigroups include all full subsemigroups of regular semigroups, cancellative monoids and Rees matrix semigroups over a cancellative monoid with the sandwich matrix entries coming from the group of units \cite{armstrong}. Using the ESN approach, Armstrong proved a category equivalence of the category of concordant semigroups with the category of inductive cancellative categories---which are generalisations of inductive groupoids. It must be mentioned here that, more recently, Armstrong's result has been further generalised to weakly $U$-regular semigroups by Wang \cite{wangureg}.

The structure of the article is as follows. In Section \ref{secconc}, we discuss concordant semigroups and study the category of principal left (right) ideals generated by their idempotents. In Section \ref{seccons}, we abstractly characterise this category as consistent category and construct an intermediary concordant semigroup arising from it. In the next section, we introduce the notion of consistent dual and cross-connections, and show how a concordant semigroup gives rise to a pair of consistent categories which are cross-connected. Section \ref{seccxn} describes the converse: how a pair of cross-connected consistent categories gives rise to a concordant semigroup. In the next section, we prove the category equivalence between the category of concordant semigroups and the category of cross-connected consistent categories. In the last section, we specialise our results to normal categories to obtain the  regular semigroup case. Also, we identify the inductive cancellative category `sitting inside' the cross-connected consistent categories, thereby describing the interrelationship between our approach and Armstrong's one.  
  
The article is structured in such a way that even a fresh reader (with no prior understanding of the cross-connections of regular semigroups) would be able to follow the discussion. Nevertheless, to place the results of this article in a proper context, we have included a brief outline of the cross-connection construction of regular semigroups as Appendix \ref{appcxnrs}. The notions undefined in the appendix shall be motivated and precisely defined in the due course of the article.

Our results reaffirm the suitability of cross-connections to describe the structure of general classes of semigroups. In fact, the third author has made the first step in generalising the cross-connection theory to arbitrary semigroups by introducing set based categories (SBCs) \cite{newcross,newcross1}. Hence, the article is an invitation for the readers to employ  cross-connection theory to do what it was built for: to overcome the limitation of not having an inbuilt left-right duality.

\section{Concordant semigroups}\label{secconc}
In this section, first we introduce concordant semigroups and then with each concordant semigroup $S$, we associate two categories $\mathbb{L}(S)$ and $\mathbb{R}(S)$ and study their properties. This will lead us to their characterisation as abstract categories. We assume familiarity with some basic notions from category theory and semigroup theory. For undefined notions, we refer to \cite{mac,higginscat} for category theory and \cite{clif,howie,grillet} for semigroups and biordered sets. Since the construction is very much similar to that of regular semigroups, often when an exact repetition of arguments suffices, we shall refer to \cite{cross}. In the sequel, all functions and morphisms shall be written in the order of their composition, i.e., from left to right.

\subsection{Generalised Green relations}
Although concordant semigroups were formally introduced and studied by Armstrong as generalisations of regular semigroups, the ideas originated from the earlier works of Fountain, El-Qallali, Lawson et al. on  generalised Green relations \cite{fountain,elqallali,lawsonabundant} in the context of abundant semigroups.

The generalised Green relations $\mathrel{\mathscr{L}}^*$ and $\mathrel{\mathscr{R}}^*$ on a semigroup $S$ are defined as follows: two elements $a$ and $b$ of a semigroup $S$ are $\mathrel{\mathscr{L}}^*$-related ($\mathrel{\mathscr{R}}^*$-related) if and only if they are $\mathrel{\mathscr{L}}$-related ($\mathrel{\mathscr{R}}$-related) in some oversemigroup of $S$. Clearly, $\mathrel{\mathscr{L}}^*$ and $\mathrel{\mathscr{R}}^*$ are right and left congruences, respectively.  
\begin{lem}\cite[Section X.1.6]{lyapin} \cite[Lemma 1.7]{mcalcancel}  Let $a$ and $b$ be elements of a semigroup $S$, then the following are equivalent: 
\begin{enumerate}[(i)]
	\item $a\mathrel{\mathscr{L}}^*b$;
	\item for all $x,y\in S^1$, $ax=ay$ if and only if $bx=by$.
\end{enumerate}
\end{lem}
The above lemma shows that any idempotent acts as a right identity within its $\mathrel{\mathscr{L}}^*$-class and further, we have the following.
\begin{lem}\cite[Corollary 1.2]{fountain}\label{lemconid} Let $e$ be an idempotent in a semigroup $S$, then for an arbitrary $a\in S$, the following are equivalent: 
	\begin{enumerate}[(i)]
		\item $a\mathrel{\mathscr{L}}^*e$;
		\item $ae=a$ and for all $x,y\in S^1$, $ax=ay$ implies $ex=ey$.
	\end{enumerate}
\end{lem}
Clearly, analogous dual results hold for the relation $\mathrel{\mathscr{R}}^*$. It can be easily seen that if a semigroup is regular, then the generalised Green relations coincide with the original Green relations. 
\begin{dfn}\cite{fountain}
A semigroup $S$ is said to be \emph{abundant} if each $\mathrel{\mathscr{L}}^*$-class and each $\mathrel{\mathscr{R}}^*$-class of $S$ contains an idempotent.
\end{dfn}
Let $S$ be an abundant semigroup with the set of idempotents $E(S)$. As in \cite{mem}, we can define quasi-orders $\mathrel{\omega}^r$ and $\mathrel{\omega}^l$ on $E(S)$ as follows:
$$e \mathrel{\omega}^l f\iff ef=e \iff Se \subseteq Sf; \text{  and  } e \mathrel{\omega}^r f \iff fe=e \iff eS \subseteq fS$$
where $e,f \in E(S)$. Then clearly the restrictions of the Green relations on the idempotents of the semigroup are given by ${\mathscr{L}}=\omega^l\cap(\omega^l)^{-1}$ and ${\mathscr{R}}=\omega^r\cap(\omega^r)^{-1}$. Also the natural partial order $\omega$ on $E(S)$ is given by $\omega=\omega^l\cap\omega^r$. For $e\in E(S)$, we denote by $\langle e\rangle$ the subsemigroup generated by the set
$$\omega(e)=\{g\in E(S)\::\: ge=eg=g\}.$$
\begin{dfn}\cite{elqallali}
A semigroup $S$ is \emph{idempotent-connected (IC)} if for each element $a\in S$ and for some $a^\dagger\in R_a^*(S) \cap E(S)$, $a^\ast\in L_a^*(S) \cap E(S)$, there is a bijection $\alpha\colon\langle a^\dagger \rangle \to \langle a^\ast \rangle$ satisfying $xa=a(x\alpha)$ for all $x\in \langle a^\dagger \rangle$.	
\end{dfn}
The above condition may be seen as a generalisation of the ample condition satisfied by ample (also known as `type A') semigroups \cite{armstrongample}. It can be seen that $\alpha$ should in fact be an isomorphism \cite{elqallali}, which shall be called as a \emph{connecting isomorphism} in the sequel. Observe that any regular semigroup is idempotent-connected \cite{elqallali}. We shall require the following lemma in the sequel, which will simplify the IC condition:
\begin{lem}\cite[Corollary to Lemma 2.3]{armstrong}\label{lemico}
A semigroup is idempotent-connected if and only if for each element $a\in S$ and for some $a^\dagger\in R_a^*(S) \cap E(S)$, $a^\ast\in L_a^*(S) \cap E(S)$, there is a unique bijection $\alpha\colon \omega( a^\dagger ) \to \omega( a^\ast )$ satisfying $xa=a(x\alpha)$ for all $x\in \omega( a^\dagger )$.	
\end{lem} 

\begin{dfn}
A \emph{concordant semigroup} is an idempotent-connected abundant semigroup whose idempotents generate a regular subsemigroup.
\end{dfn}

The homomorphic image of an abundant semigroup is not necessarily abundant. So, a homomorphism $\phi\colon S \to T$ of semigroups is defined to be a \emph{good homomorphism} if for any $a,b \in S$, $a\mathrel{\mathscr{L}}^*b$ in $S$ implies $a\phi\mathrel{\mathscr{L}}^*b\phi$ in $T$ and $a\mathrel{\mathscr{R}}^*b$ in $S$ implies $a\phi\mathrel{\mathscr{R}}^*b\phi$ in $T$ \cite{elqallali}. Then as shown in \cite[Theorem 2.5]{armstrong}, a `good homomorphic' image of a concordant semigroup is concordant. Hence concordant semigroups with good homomorphisms as morphisms form a category, say $\mathbf{CS}$. It has been shown in \cite{romeocal,romeo} that the category $\mathbf{RS}$ of regular semigroups is a reflective subcategory of the category $\mathbf{CS}$.

\subsection{Categories from a concordant semigroup}
To extend the cross-connection analysis to concordant semigroups, we need to identify proper generalisations of normal categories of the regular case. This quest leads us to the category $\mathbb{L}(S)$ that arises from the principal left ideals generated by the idempotents of a concordant semigroup. In the sequel, $S$ shall denote a concordant semigroup and  $E(S)$ its set of idempotents. The set of objects of the category $\mathbb{L}(S)$ is given by 
$$v\mathbb{L}(S) = \{ Se : e \in E(S) \}.$$
For each $x\in Se$, a morphism from $Se$ to $Sf$ is the function $\rho(e,u,f)\colon x \mapsto xu $ where $u\in eSf$. Thus a morphism is a partial right translation (i.e. a right translation restricted to a principal left ideal). Then, as in \cite[Lemma III.12]{cross}, we can easily see that the morphisms $\rho(e,u,f)=\rho(g,v,h)$ if and only if $e\mathrel{\mathscr{L}}g$, $f\mathrel{\mathscr{L}}h$ and $v=gu$. Also, given any two morphisms, say $\rho(e,u,f)$ and $\rho(g,v,h)$, they are composable if $Sf=Sg$ (i.e., if $f\mathrel{\mathscr{L}}g$) and then
$$\rho(e,u,f)\:\rho(g,v,h) = \rho(e,uv,h).$$
It is clear that $\mathbb{L}(S)$ forms a small category such that $\rho(e,e,e)$ is the identity morphism at the vertex $Se$ and $\mathbb{L}(S)$ is a subcategory of the category $\mathbf{Set}$. 
Thus the set of all morphisms in  the category $\mathbb{L}(S)$ from the object $Se$ to $Sf$ is given by the set $ \{ \rho(e,u,f) : u\in eSf \}$.

Recall that a morphism in a category is called a \emph{monomorphism} if it is right cancellable; an \emph{epimorphism} if it is left cancellable; and a \emph{bimorphism} if it is both right and left cancellable. A morphism $f\colon c\to d$ in a category $\mathcal{C}$ is said to be an \emph{isomorphism} if there exists a morphism $g\colon d\to c$ in $\mathcal{C}$ such that $fg=1_c$ and $gf=1_d$. Clearly, an isomorphism is a bimorphism, but not conversely. The next lemma characterises the morphisms in the category $\mathbb{L}(S)$.

\begin{lem}\label{lemmor}
Let $S$ be a concordant semigroup and $\rho(e,u,f)$ be an arbitrary morphism in the category $\mathbb{L}(S)$. Then
\begin{enumerate}[(i)]
\item $\rho(e,u,f)$ is a monomorphism if and only if $e\mathrel{\mathscr{R}}^*u$;
\item $\rho(e,u,f)$ is an epimorphism if and only if $u\mathrel{\mathscr{L}}^*f$;
\item $\rho(e,u,f)$ is a bimorphism if and only if $e\mathrel{\mathscr{R}}^*u\mathrel{\mathscr{L}}^*f$.
\end{enumerate}
\end{lem}
\begin{proof}
Let  $e\mathrel{\mathscr{R}}^*u$. Then suppose $\rho(e',x,e)\rho(e,u,f)= \rho(e'',y,e)\rho(e,u,f)$, i.e., the morphism  $\rho(e',xu,f)=\rho(e'',yu,f)$. Then $e'\mathrel{\mathscr{L}}e''$ and $xu=e'yu$. Since $S$ is concordant, using the dual result of Lemma \ref{lemconid}, we have $xe=e'ye$. But since the elements $x,e'y \in Se$, we get $x=e'y$. So $\rho(e',x,e)= \rho(e'',y,e)$ and hence $\rho(e,u,f)$ is right cancellable.\\
Conversely, if $\rho(e,u,f)$ is right cancellable, since $S$ is concordant there exists $g'\in E(S)$ such that $g'\mathrel{\mathscr{R}}u$. Then $g'u=u=eu$ will imply $g'=eg'$  (by Lemma \ref{lemconid}). That is $g'\mathrel{\omega}^r e$. If we let $g=g'e$, then $g'\mathrel{\mathscr{R}}g\mathrel{\omega} e$ and $g\mathrel{\mathscr{R}}g'\mathrel{\mathscr{R}}^*u$, hence $g\mathrel{\mathscr{R}}^*u$. Since the morphism $\rho(e,u,f)=\rho(e,g,g)\rho(g,u,f)$ is a monomorphism, we have $\rho(e,g,g)$ is a monomorphism.  Therefore, from
$$\rho(e,g,e)\rho(e,g,g) =\rho(e,g,g)=\rho(e,e,e)\rho(e,g,g),$$
we have $\rho(e,e,e)=\rho(e,g,e)$ which implies $e=g$. Hence $e\mathrel{\mathscr{R}}^*u$.\\
Similarly, we can prove $(ii)$; $(iii)$ follows from $(i)$ and $(ii)$.   
\end{proof}

Now, we define a subcategory $\mathcal{P}_\mathbb{L}$ of the category $\mathbb{L}(S)$ such that $v\mathcal{P}_\mathbb{L}=v\mathbb{L}(S)$ and whenever $Se\subseteq Sf$, there is a unique morphism $j^{Sf}_{Se}=\rho(e,e,f) \in \mathcal{P}_\mathbb{L}$. The morphisms of the subcategory $\mathcal{P}_\mathbb{L}$ shall be called \emph{inclusions} as they correspond to the inclusions of the principal ideals. By definition, $\mathcal{P}_\mathbb{L}$ is a strict preorder category, i.e., a preorder category in which the identity morphisms are the only isomorphisms. Clearly, every inclusion is a monomorphism. Also for morphisms $\rho(e,e,f),\rho(g,g,f)\in \mathcal{P}_\mathbb{L}$ such that $\rho(e,e,f)=\rho(h,u,g)\rho(g,g,f)$ in the category $\mathbb{L}(S)$, then $\rho(e,e,f)=\rho(h,u,f)$ so that $u=he=h$; hence the morphism $\rho(h,u,g)=\rho(h,h,g)\in \mathcal{P}_\mathbb{L}$.

This leads us to the following definition.
\begin{dfn}\label{catsub}
Let $\mathcal{C}$ be a small category and $\mathcal{P}$ be a subcategory of $\mathcal{C}$. Then the pair $(\mathcal{C},\mathcal{P})$ (often denoted by just $\mathcal{C}$) is said to be a \emph{category with subobjects} if:
\begin{enumerate}[(i)]
\item $\mathcal{P}$ is a strict preorder with $v\mathcal{P}=v\mathcal{C}$.
\item Every $f\in \mathcal{P}$ is a monomorphism in $\mathcal{C}$.
\item If $f,g\in \mathcal{P}$ and if $f=hg$ for some $h\in \mathcal{C}$, then $h\in \mathcal{P}$.
\end{enumerate}
\end{dfn}
Since a small strict preorder is equivalent to a partially ordered set (poset), the above definition characterises a category whose object set is a poset and with certain distinguished morphisms arising from the comparability of the poset. 

Observe that for an inclusion $\rho(e,e,f)\in\mathbb{L}(S)$, since $Se\subseteq Sf$, we have $ef=e$ and $fe\in fSe$ so that $$\rho(e,e,f)\rho(f,fe,e)=\rho(e,e(fe),e)=\rho(e,(ef)e,e)=\rho(e,e,e).$$
So, every inclusion in the category $\mathbb{L}(S)$ \emph{splits}, i.e., has a right inverse. These right inverses shall be called \emph{retractions}. 
The following lemma characterises the retractions in $\mathbb{L}(S)$. 
\begin{lem}\label{lemret}
Let $\rho(e,e,f)$ be an inclusion such that $Se\subseteq Sf$. Then $\rho(f,x,e)$ is a retraction if and only if $x\in E(L_e)\cap\omega(f)$.
\end{lem}
\begin{proof}
Suppose the morphism $\rho(f,x,e)$ is a retraction such that $x\in fSe\subseteq Se$. Then we have $\rho(e,e,f)\rho(f,x,e)=1_{Se}$, so
$$x=x\rho(e,e,f)\rho(f,x,e)=x\rho(e,ex,e)=x(ex)=(xe)x=x^2.$$
From above, we have $ex=e$ and $xe=e$, so $x\in E(L_e)$. Also, since $x\in Se\subseteq Sf$ and $x\in fSe\subseteq fS$, we have $x\in \omega(f)$. Hence $x\in E(L_e)\cap\omega(f)$. The converse is clear. 
\end{proof}

Now, we proceed to discuss a special factorisation property of the morphisms in the category $\mathbb{L}(S)$. Let $\rho(e,u,f)$ be an arbitrary morphism in $\mathbb{L}(S)$, then since $S$ is concordant, there exist $g',h'\in E(S)$ such that $g'\mathrel{\mathscr{R}}^*u\mathrel{\mathscr{L}}^*h'$. Since $g'\mathrel{\mathscr{R}}u$, $g'u=u=eu$ and so $g'g'=eg'$, i.e., $g'=eg'$ or $g'\mathrel{\omega}^r e$. So $g'e\mathrel{\mathscr{R}}g'\mathrel{\mathscr{R}}^*u$ and $g'e\mathrel{\omega} e$. If we let $g=g'e$, then $g\mathrel{\mathscr{R}}^*u$ and $g\mathrel{\omega} e$. Then by Lemma \ref{lemret}, $\rho(e,g,g)$ is a retraction from $Se$ to $Sg$.

Similarly, since $u\mathrel{\mathscr{L}}^*h'$, $uh'=u=uf$ and if we let $h=fh'$, then $h\mathrel{\mathscr{L}}h'\mathrel{\mathscr{L}}^*u$ and $h\mathrel{\omega} f$. So we have $g\mathrel{\mathscr{R}}^*u\mathrel{\mathscr{L}}^*h$ and by Lemma \ref{lemmor}, $\rho(g,u,h)$ is a bimorphism in $\mathbb{L}(S)$. Also, $\rho(h,h,f)$ is clearly an inclusion. Also,
$$\rho(e,g,g,)\rho(g,u,h)\rho(h,h,f)=\rho(e,guh,f)= \rho(e,u,f).$$
Hence any morphism in the category $\mathbb{L}(S)$ has a factorisation of the above form, composed of a retraction, a bimorphism and an inclusion. This factorisation, which can be illustrated using the following diagram, is indeed a characterising property.
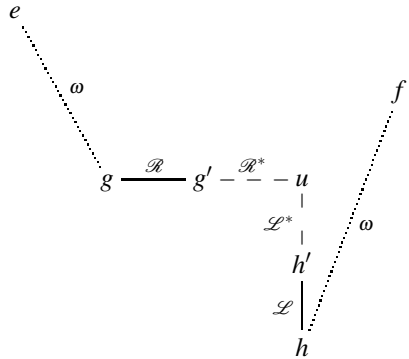
\begin{equation*}
	\xymatrixcolsep{2pc}\xymatrixrowsep{1.5pc}
	\xymatrix{ 
		e\ar@{.}[ddr]^{\mathrel{\omega}}\\
		&&&&f\ar@{.}[dddl]^{\mathrel{\omega}}\\
		&g\ar@{-}[r]^{\mathrel{\mathscr{R}}}&g'\ar@{--}[r]^{\mathrel{\mathscr{R}}^*}&u\ar@{--}[d]_{\mathrel{\mathscr{L}}^*}\\
		&&&h'\ar@{-}[d]_{\mathrel{\mathscr{L}}}\\
		&&&h
}
\end{equation*}
\begin{dfn}\label{consfact}
Let $\mathcal{C}$ be a category with subobjects. Then a morphism $f$ in $\mathcal{C}$ is said to have a \emph{consistent factorisation} if $f=quj$, where $q$ is a retraction, $u$ is a bimorphism and $j$ is an inclusion, respectively in $\mathcal{C}$.
\end{dfn}
Then the morphism $qu$ is known as the epimorphic component of the morphism $f$ and shall be denoted in the sequel by $f^\circ$. The codomain of $f^\circ$ is called the image of $f$ and shall be denoted as $im f$. It can be seen that although a morphism can have different consistent factorisations, the epimorphic component and hence the image is unique \cite[Section II.2]{cross}. 

Thus, normal factorisation of a morphism as defined below can be seen as a specialised consistent factorisation.
\begin{dfn}\label{normfact}
Let $\mathcal{C}$ be a category with subobjects. Then a morphism $f$ in $\mathcal{C}$ is said to have a \emph{normal factorisation} if $f=quj$, where $q$ is a retraction, $u$ is an isomorphism and $j$ is an inclusion, respectively in $\mathcal{C}$.
\end{dfn}

Now, we shift our focus to the idempotent-connectedness property of $S$. But for that, we need to extend the notion of an order ideal of a poset, to categories with subobjects, by identifying certain special subcategories. Let $\mathcal{C}$ be a category with subobjects and $c\in v\mathcal{C}$, we define an \emph{ideal} $( c )$ of $\mathcal{C}$ as the full subcategory of $\mathcal{C}$ whose objects are subobjects of $c$ in $\mathcal{C}$. Also we shall denote by $\sigma\mathcal{C}$, the preorder of subobjects of $\mathcal{C}$ (i.e., the subcategory such that $v\sigma\mathcal{C}=v\mathcal{C}$ and morphisms of $\sigma\mathcal{C}$ are all inclusions in $\mathcal{C}$) and by $cor\mathcal{C}$ the subcategory of $\mathcal{C}$ generated by inclusions and retractions. In particular for an object $c\in v\mathcal{C}$, $\sigma(c)$ will denote the preorder whose objects are subobjects of $c$ and $\langle c \rangle$ shall denote the full subcategory of $cor\mathcal{C}$ generated by the inclusions and retractions among the subobjects~of~$c$. Observe that for any $c$ in $\mathcal{C}$, 
$$\sigma(c)\subseteq\langle c\rangle \subseteq(c).$$

Suppose $\rho=\rho(e,u,f)$ is a bimorphism in $\mathbb{L}(S)$, then by Lemma \ref{lemmor},  $e\mathrel{\mathscr{R}}^*u\mathrel{\mathscr{L}}^*f$. Then $\rho$ defines a functor $T^\rho$ between the sub-preorders $\sigma(Se) $ and $ \sigma(Sf) $ of the category $\mathbb{L}(S)$ as follows. For each $Sg\subseteq Sh\subseteq Se$,
\begin{equation*}
T^\rho(Sg) = im(j_{Sg}^{Se}\rho) \text{ and }T^\rho(j_{Sg}^{Sh})=j_{T^\rho(Sg)}^{T^\rho(Sh)}.
\end{equation*} 
Recall that $\sigma(Se)\subseteq\mathbb{L}(S)$ so that $J=J(\sigma(Se),\mathbb{L}(S))$ is an inclusion functor. Since $\sigma(Sf)\subseteq\mathbb{L}(S)$, the functor $T^\rho$ can also be realised as a functor from $\sigma(Se)$ to $\mathbb{L}(S)$. Now, for each $Sg\subseteq Se$, if we associate
$$\bar{\rho}\colon Sg\mapsto(j_{Sg}^{Se}\rho)^\circ$$
then we can see that the following diagram commutes for all $Sg \subseteq Sh\subseteq Se$:
\begin{equation*}\label{}
\xymatrixcolsep{3pc}\xymatrixrowsep{4pc}\xymatrix
{
	Sh \ar[rr]^{\bar{\rho}(Sh)}   	&& T^\rho(Sh)  \\       
	Sg \ar[u]^{j_{Sg}^{Sh}}\ar[rr]^{\bar{\rho}(Sg)} && T^\rho(Sg)\ar[u]_{T^\rho(j_{Sg}^{Sh})} 
	}
\end{equation*}
That means $\bar{\rho}$ is a natural transformation between the functors $J$ and $T^\rho$.   

Now, since $S$ is a concordant semigroup, we know that there exists a connecting isomorphism $\alpha\colon \langle e\rangle \to \langle f \rangle$. Using this connecting isomorphism, we can define a new functor $T^\alpha\colon\langle Se\rangle \to \langle Sf \rangle$ as follows. For each $Sg \subseteq Sh\subseteq Se$ and $\rho(g,u,h)\in \langle Se\rangle$,
\begin{equation*}
T^\alpha(Sg) = S(g\alpha) \text{ and }T^\alpha(\rho(g,u,h))=\rho(g\alpha,u\alpha,h\alpha).
\end{equation*} 
Since $\alpha$ is a semigroup isomorphism, we can easily verify that $T^\alpha$ is a category isomorphism. Then $T^\alpha$ is an extension of $T^\rho$ and for each $Sg\subseteq Se$, if we associate
$$\bar{\alpha}\colon Sg\mapsto T^\alpha(Sg),$$
then we can see that the following diagram commutes for all $Sg \subseteq Sh\subseteq Se$:
\begin{equation*}\label{}
\xymatrixcolsep{3pc}\xymatrixrowsep{4pc}\xymatrix
{
	Sh \ar[rr]^{\bar{\alpha}(Sh)}   	&& T^\alpha(Sh)  \\       
	Sg \ar[u]^{\rho(g,u,h)}\ar[rr]^{\bar{\alpha}(Sg)} && T^\alpha(Sg)\ar[u]_{T^\alpha(\rho(g,u,h))} 
}
\end{equation*}
Since for each $Sg\subseteq Se$, the morphism $\bar{\alpha}(Sg)$ is an isomorphism, we see that $\bar{\alpha}$ is in fact a natural isomorphism between the inclusion functor $J(\langle Se\rangle, \mathbb{L}(S))$ and the functor $T^\alpha \colon \langle Se\rangle\to \langle Sf\rangle\subseteq\mathbb{L}(S)$. Observe that since $T^{\alpha}$ is an extension of $T^{\rho}$, the inclusions in categories $\sigma(Se)$ and $\sigma(Sf)$ split in the categories $\langle Se \rangle$ and $\langle Sf \rangle$, respectively. Hence, we may also have an equivalent diagram as above with the vertical arrows pointing downwards, corresponding to the retractions in $\langle Se\rangle$ and $\langle Sf\rangle$.
 
Summarising the above discussion: given any bimorphism $\rho(e,u,f)$ in the category $\mathbb{L}(S)$, the functor $T^\rho$ can be extended to an isomorphism $T^\alpha\colon\langle Se\rangle \to \langle Sf\rangle$ such that $\bar{\alpha}$ is a natural isomorphism. 

Further, it can be seen that if $T^\alpha$ and $T^{\alpha'}$ are any two extensions of the functor $T^\rho$ with $\bar{\alpha}$ and $\bar{\alpha}'$ natural isomorphisms, then $T^\alpha=T^{\alpha'}$. So, an extension of $T^\rho$ with the above property is unique. Hence, we have the following definition.
\begin{dfn}
A bimorphism $u\colon c\to d$ in a category $\mathcal{C}$ is said to be \emph{consistent} if the functor $T^u\colon\sigma(c)\to\sigma(d)$ defined as: 
$$T^u(c')=im(j_{c'}^{c}u) \text{ and }T^u(j_{c'}^{c''})=j_{T^u(c')}^{T^u(c'')} \ \text{ for each }c'\subseteq c'' \subseteq c,$$
can be extended uniquely to an isomorphism $T\colon \langle c\rangle \to \langle d\rangle$ such that  the map given by $c'\mapsto T(c')$  is a natural isomorphism between the functors $J(\langle c \rangle,\mathcal{C})$ and $T\colon\langle c \rangle\to\mathcal{C}$.
\end{dfn}

We know that the idempotents $E(S)$ of a concordant semigroup $S$ generate a regular subsemigroup, i.e., the biordered set $E(S)$ is regular \cite{mem}. Also recall from \cite{mem} that a biordered set $E(S)$ is regular if and only if the sandwich set 
$$ \mathcal{S}(e,f)=\{h\in E(S) : ehf=ef \text{ and } fhe=h\}$$ 
is non empty for each pair of idempotents $e,f\in E(S)$. 

Now, consider a morphism $\rho$ in the category $\mathbb{L}(S)$ such that it is a product of an inclusion and a retraction. So, $$\rho=\rho(e,e,f)\rho(f,g,g)=\rho(e,eg,g)$$ 
such that $Se\subseteq Sf$ and $g\mathrel{\omega} f$.
Then let $h\in \mathcal{S}(e,g)$ so that in the regular biordered set $E(S)$, we have $eh\mathrel{\mathscr{L}}h\mathrel{\mathscr{R}}hg\mathrel{\mathscr{L}}eg\mathrel{\mathscr{R}}eh$ with $eh\mathrel{\omega} e$ and $hg\mathrel{\omega} g$, as shown below. 
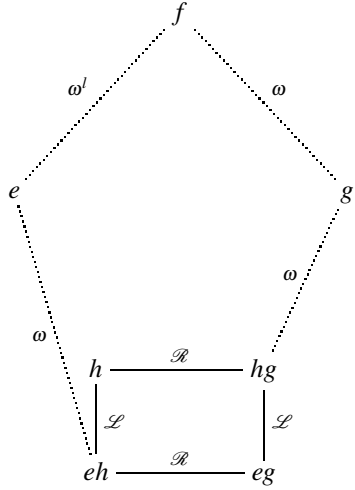
\begin{equation*}
	\xymatrixcolsep{1.5pc}\xymatrixrowsep{2pc}
	\xymatrix{ 
		&&f\ar@{.}[ddll]_{\mathrel{\omega}^l}\ar@{.}[ddrr]^{\mathrel{\omega}}\\\\
		e\ar@{.}[dddr]_{\mathrel{\omega}}&&&&g\ar@{.}[ddl]_{\mathrel{\omega}}\\\\
		&h\ar@{-}[rr]^{\mathrel{\mathscr{R}}}\ar@{-}[d]^{\mathrel{\mathscr{L}}}&&hg\ar@{-}[d]^{\mathrel{\mathscr{L}}}\\
		&eh\ar@{-}[rr]^{\mathrel{\mathscr{R}}}&&eg
	}
\end{equation*}

Hence,
\begin{align}\label{normret}
\begin{split}
\rho(e,eg,g)=&\rho(e,ehg,g)\\
=&\rho(e,e(he)(gh)g,g)\\
=&\rho(e,(eh)(eg)(hg),g)\\
=&\rho(e,eh,eh)\rho(eh,eg,hg)\rho(hg,hg,g).
\end{split}
\end{align}
Also since 
$$\rho(eh,eg,hg)\rho(hg,h,eh)=\rho(eh,eh,eh) \text{ and }\rho(hg,h,eh)\rho(eh,eg,hg)=\rho(hg,hg,hg),$$
the morphism $\rho(eh,eg,hg)$ is an isomorphism and so the morphism $\rho=\rho(e,eg,g)$ has a normal factorisation of the above form (\ref{normret}). Thus, every morphism in the category $\mathbb{L}(S)$, which is a product of an inclusion  and a  retraction, admits a normal factorisation.

Now, let $a$ be an arbitrary element of $S$, and for each $Se\in v\mathbb{L}(S)$, define a function $\rho^a\colon v\mathbb{L}(S)\to \mathbb{L}(S)$ as follows:
\begin{equation}\label{eqnprinc}
\rho^a(Se) = \rho(e,ea,f)\text{ where } f\in E(L^*_a).
\end{equation}
Then for $Se'\subseteq Se$, the inclusion morphism $j_{Se'}^{Se} = \rho(e',e',e)$ and so
$$j_{Se'}^{Se}\rho^a(Se)=\rho(e',e',e)\rho(e,ea,f)=\rho(e',e'a,f)=\rho^a(Se').$$ 
Further, since $S$ is abundant, there exists $g\in E(R^*_a)$ such that $\rho^a(Sg)=\rho(g,ga,f)=\rho(g,a,f)$ is a bimorphism (by Lemma \ref{lemmor}). Hence we define the following:
\begin{dfn}
Let $\mathcal{C}$ be a category with subobjects and $d\in v\mathcal{C}$. Then for each $c\in v\mathcal{C}$, a function $\gamma\colon a\mapsto\gamma(a)\in\mathcal{C}(a,d)$ from $v\mathcal{C}$ to $\mathcal{C}$ is said to be a \emph{consistent cone (respectively normal cone)} with apex $d$ if:
\begin{enumerate}
	\item whenever $a\subseteq b$, $j_a^b \gamma(b)=\gamma(a)$;
	\item there exists at least one $c\in v\mathcal{C}$ such that $\gamma(c)\colon c\to d$ is a bimorphism (respectively isomorphism).
\end{enumerate}
\end{dfn}
Then for a consistent cone $\gamma$, we denote by $c_\gamma$ the apex of $\gamma$ and the morphism $\gamma(c)$ is called the component of the cone $\gamma$ at the apex $c$. Since every isomorphism is a bimorphism, observe that every normal cone is a consistent cone. 

Hence, from the above discussion, we can see that $\rho^a$ is a consistent cone with apex $Sf$. In the sequel, the consistent cone $\rho^a$ shall be called the \emph{principal cone} determined by the element $a$. In particular, observe that, for an idempotent $e\in E(S)$, we have a principal cone $\rho^e(Se)= \rho(e,e,e)=1_{Se}$. Hence for each object $Se\in v\mathbb{L}(S)$, there exists a consistent cone such that its component at $Se$ is the identity morphism. This is a reflection of the abundance condition. 

\section{Consistent categories}\label{seccons}
Now, we proceed to define consistent categories as the abstractions of the category $\mathbb{L}(S)$ of the principal left ideals generated by the idempotents of a concordant semigroup $S$. 
\begin{dfn}
A category $\mathcal{C}$ is said to be a \emph{consistent category} if:
\begin{enumerate}
	\item [(CC 1)] $\mathcal{C}$ is a category with subobjects;
	\item [(CC 2)] every inclusion in $\mathcal{C}$ splits;
	\item [(CC 3)] every morphism in $\mathcal{C}$ admits a consistent factorisation;
	\item [(CC 4)] every bimorphism is consistent;
	\item [(CC 5)] if $f\in\mathcal{C}$ such that $f=jq$ where $j$ is an inclusion and $q$ is a retraction, then $f$ admits a normal factorisation;
	\item [(CC 6)] for each $c\in v\mathcal{C}$ there exists a consistent cone $\epsilon$ such that $\epsilon(c)=1_c$.
\end{enumerate}
\end{dfn}

Also recall the following definition of a normal category which is an abstraction of the principal (left) ideals of a regular semigroup. Notice that the term `normal category' has been used in several other non-related senses in the literature. Nevertheless, we keep this term as introduced in \cite{cross}.
\begin{dfn}\cite[Section III.1.3]{cross}\label{dfnnormc}
	A category $\mathcal{C}$ is said to be a \emph{normal category} if:
	\begin{enumerate}
		\item [(NC 1)] $\mathcal{C}$ is a category with subobjects;
		\item [(NC 2)] every inclusion in $\mathcal{C}$ splits;
		\item [(NC 3)] every morphism in $\mathcal{C}$ admits a normal factorisation;
		\item [(NC 4)] for each $c\in v\mathcal{C}$ there exists a normal cone $\epsilon$ such that $\epsilon(c)=1_c$.
	\end{enumerate}
\end{dfn}

Observe that by \cite[Corollary II.8]{cross}, in a given normal category~$\mathcal{C}$, every bimorphism $f\colon c\to d$  is an isomorphism. Then $T^f\colon\langle c\rangle\to \langle d\rangle$ is an isomorphism in $cor \mathcal{C}$. So every bimorphism in $\mathcal{C}$ is consistent and thus every normal category is a consistent category. 

The discussion is Section \ref{secconc} shows that $\mathbb{L}(S)$ is indeed a consistent category when $S$ is a concordant semigroup. Now, we proceed to show that in fact every consistent category arises as $\mathbb{L}(S)$ for some concordant semigroup $S$. For this end, we need to associate a concordant semigroup with a given consistent category $\mathcal{C}$; naturally we look for that semigroup in the set of all consistent cones in $\mathcal{C}$.

Let $\mathcal{C}$ be a consistent category and let $\gamma$ be a consistent (normal) cone in $\mathcal{C}$, if $f\in \mathcal{C}(c_\gamma,d)$ be an epimorphism, then as in \cite[Lemma I.1]{cross}, we can easily see that the map
$$\gamma\ast f\colon c \mapsto \gamma(c)f \text{ for all } c\in v\mathcal{C}	$$ 
is a consistent (respectively normal) cone such that $c_{\gamma\ast f}=d$. Hence for $\gamma ^{(1)},\gamma ^{(2)}\in  \widehat{\mathcal{C}}$,
\begin{equation}\label{eqnbin}
\gamma ^{(1)}\cdot\gamma ^{(2)}=\gamma ^{(1)}\ast(\gamma ^{(2)}(c_{\gamma ^{(1)}}))^\circ
\end{equation}
where $(\gamma ^{(2)}(c_{\gamma ^{(1)}}))^\circ$ is the epimorphic component of the morphism $\gamma ^{(2)}(c_{\gamma ^{(1)}})$, defines a binary composition on the set of all consistent cones in $\mathcal{C}$. The following lemma directly follows from \cite[Theorem I.2]{cross}.
\begin{lem}\label{lemconssemi}
Let $\mathcal{C}$ be a consistent category. Then the set of all consistent cones forms a semigroup under the binary composition defined in (\ref{eqnbin}). A consistent cone $\epsilon$ in $\mathcal{C}$ is an idempotent if and only if $\epsilon(c_\epsilon)=1_{c_\epsilon}$.
\end{lem}
Although the set of all consistent cones forms a semigroup, it need not necessarily be concordant. But it has a suitable subsemigroup $ \widehat{\mathcal{C}}$ which will serve our purpose.
\begin{pro}
Let $\mathcal{C}$ be a consistent category and let $ \widehat{\mathcal{C}}$ denote the set of all consistent cones $\gamma$ in $\mathcal{C}$ such that $\gamma=\epsilon\ast u$ where $\epsilon$ is an idempotent cone and $u$ is a bimorphism in $\mathcal{C}$. Then the set $ \widehat{\mathcal{C}}$ is a semigroup under the binary composition defined in (\ref{eqnbin}).
\end{pro}
\begin{proof}
In the light of Lemma \ref{lemconssemi}, we just need to show that $ \widehat{\mathcal{C}}$ is closed. Let $\gamma ^{(1)}=\epsilon ^{(1)}\ast u_1$ and $\gamma ^{(2)}=\epsilon ^{(2)}\ast u_2$ where $\epsilon ^{(1)},\epsilon ^{(2)}$ are idempotent cones and $u_1\colon c_{\epsilon ^{(1)}}\to c_1$, $u_2\colon c_{\epsilon ^{(2)}}\to c_2$ are bimorphisms. Then 
\begin{align*}\label{}
\gamma ^{(1)}\cdot\gamma ^{(2)}&=(\epsilon ^{(1)}\ast u_1)(\epsilon ^{(2)}\ast u_2)\\
&=(\epsilon ^{(1)}\ast u_1)\ast((\epsilon ^{(2)}\ast u_2)(c_{1}))^\circ\\
&=\epsilon ^{(1)}\ast (u_1\epsilon ^{(2)}(c_{1}) u_2)^\circ.
\end{align*}
Let $qu$ be the consistent factorisation of the epimorphism $(u_1\epsilon ^{(2)}(c_{1}) u_2)^\circ$ and so 
$$\gamma ^{(1)}\cdot\gamma ^{(2)}=\epsilon\ast qu= (\epsilon\ast q) \ast u.$$
Now, let the codomain of the retraction $q$ be $c$ so that $c\subseteq c_{\epsilon ^{(1)}}$ and $\epsilon(c)=j_c^{c_{\epsilon ^{(1)}}}\epsilon(c_{\epsilon ^{(1)}})=j_c^{c_{\epsilon ^{(1)}}}1_{c_{\epsilon ^{(1)}}}=j_c^{c_{\epsilon ^{(1)}}}$. Then the component $\epsilon\ast q(c)=\epsilon(c) q= j_{c}^{c_{\epsilon ^{(1)}}}q=1_c$. Hence the consistent cone $\epsilon\ast q$ is an idempotent. Also since $u$ is a bimorphism, $ \widehat{\mathcal{C}}$ is closed.
\end{proof}
To show that $ \widehat{\mathcal{C}}$ is concordant, we need to first show that the idempotents of $ \widehat{\mathcal{C}}$ generate a regular subsemigroup or equivalently, identify a full regular subsemigroup of the semigroup $ \widehat{\mathcal{C}}$, such that the biordered sets of $\widehat{\mathcal{C}}$ and its subsemigroup, are isomorphic. 
\begin{lem}\label{lemnorm}
Let ${\overline{\mathcal{C}}}$ denote the set of all morphisms with normal factorisations in the consistent category $\mathcal{C}$. Then ${\overline{\mathcal{C}}}$ forms a normal subcategory of the category $\mathcal{C}$.	
\end{lem}
\begin{proof}
By Definition \ref{dfnnormc}, it is clear that if we show ${\overline{\mathcal{C}}}$ is closed under composition of morphisms, then we are done. Let $f,g$ be morphisms of the consistent category $\mathcal{C}$ such that $f=q_1u_1j_1$ and $g=q_2u_2j_2$ are normal factorisations.
Then by axiom $(CC\: 5)$, the morphism $j_1q_2$ has a normal factorisation such that $j_1q_2=q_3u_3j_3$. So,
$$fg=q_1u_1j_1\:q_2u_2j_2=q_1(u_1q_3)u_3(j_3u_2)j_2.$$
Now, using \cite[Corollary II.11]{cross}, \cite[Corollary II.10]{cross} and \cite[Proposition II.9]{cross} sequentially, we see that the epimorphism $u_1q_3$ and the monomorphism $j_3u_2$ have normal factorisations $u_1q_3=q_4u_4$ and $j_3u_2=u_5j_5$ so that
$$fg= q_1(q_4u_4)u_3(u_4j_4)j_2=(q_1q_4)(u_4u_3u_5)(j_5j_2).$$
Hence $fg$ has a normal factorisation of the above form where $q_1q_4$ is a retraction $u_4u_3u_5$ is an isomorphism and $j_5j_2$ is an inclusion. Hence the lemma.
\end{proof}

Observe that every consistent cone in ${\overline{\mathcal{C}}}$ is normal and every idempotent cone in $\mathcal{C}$ is also normal. Then by \cite[Theorem III.2]{cross}, we have the following.
\begin{pro}\label{pronormal}
The semigroup $\widehat{\overline{\mathcal{C}}}$ of all normal cones in ${\overline{\mathcal{C}}}$ is a full regular subsemigroup of the semigroup $ \widehat{\mathcal{C}}$.
\end{pro}
The following lemma regarding the biorder relations in $\hat{\mathcal{C}}$ can be easily verified.
\begin{lem}\label{lemgc}
Let $\epsilon ^{(1)},\epsilon ^{(2)}$ be idempotents in the semigroup $\hat{\mathcal{C}}$. Then $$\epsilon ^{(1)}\mathrel{\omega}^l\epsilon ^{(2)}\text{ if and only if }c_{\epsilon ^{(1)}}\subseteq c_{\epsilon ^{(2)}};$$
$$\epsilon ^{(1)}\mathrel{\omega}^r\epsilon ^{(2)}\text{ if and only if }\epsilon ^{(1)}(c_{\epsilon ^{(2)}})\text{ is an epimorphism such that }\epsilon ^{(1)}=\epsilon ^{(2)}\ast\epsilon ^{(1)}(c_{\epsilon ^{(2)}}).$$ 
\end{lem} 
Observe that the set of idempotents $E(\widehat{\overline{\mathcal{C}}})=E(\widehat{{\mathcal{C}}})$. Also by \cite[Proposition III.5]{cross} and \cite[Proposition III.7]{cross}, we can see that the quasi orders coincide. Hence the biordered sets of $E(\widehat{\overline{\mathcal{C}}})$ and $E(\widehat{{\mathcal{C}}})$ are equal. In particular, $E(\widehat{{\mathcal{C}}})$ is a regular biordered set with quasi orders defined as above.

Now, we proceed to show that $ \widehat{\mathcal{C}}$ is an abundant semigroup.
\begin{lem}\label{lemab}
Let $\mathcal{C}$ be a consistent category and let $\gamma =\epsilon\ast u \in  \widehat{\mathcal{C}}$. If $\delta$ is an idempotent cone in $\mathcal{C}$ such that $c_\gamma=c_\delta$, then $\epsilon\mathrel{\mathscr{R}}^*\gamma\mathrel{\mathscr{L}}^*\delta$. 
\end{lem}
\begin{proof}
Since $\gamma=\epsilon\ast u$, we have $\epsilon\cdot\gamma=\gamma$. Let $\gamma ^{(1)},\gamma ^{(2)}	\in \widehat{\mathcal{C}}$ be such that $\gamma ^{(1)}\gamma=\gamma ^{(2)}\gamma$. Then
\begin{align*}
\gamma ^{(1)}\cdot\epsilon\ast u&=\gamma ^{(2)}\cdot\epsilon\ast u \\
\gamma ^{(1)}\ast (\epsilon(c_{\gamma ^{(1)}}))^\circ u&=\gamma ^{(2)}\ast (\epsilon(c_{\gamma ^{(2)}}))^\circ u \\
\gamma ^{(1)}\ast (\epsilon(c_{\gamma ^{(1)}}))^\circ &=\gamma ^{(2)}\ast (\epsilon(c_{\gamma ^{(2)}}))^\circ \qquad \text{ (since $u$ is a bimorphism)}\\
\gamma ^{(1)}\cdot\epsilon&=\gamma ^{(2)}\cdot\epsilon.
\end{align*}
Hence by Lemma \ref{lemconid}, we have $\gamma\mathrel{\mathscr{R}}^*\epsilon$.

Now, since $\delta$ is an idempotent cone (such a $\delta$ exists in $\mathcal{C}$ by axiom $(CC\:6)$) and $c_\gamma=c_\delta$, we have $\delta(c_\gamma)=1_{c_\gamma}$. For any $c\in \mathcal{C}$, we have $\gamma\cdot\delta(c)= \gamma(c) (\delta(c_\gamma))^\circ=\gamma(c)$. Hence $\gamma\cdot\delta=\gamma$.
Now, if $\gamma ^{(1)},\gamma ^{(2)}\in \widehat{\mathcal{C}}$ with $\gamma\gamma ^{(1)}=\gamma\gamma ^{(2)}$, then for any $c\in v\mathcal{C}$, we have that $\gamma(c)(\gamma ^{(1)}(c_\gamma))^\circ= \gamma(c)(\gamma ^{(2)}(c_\gamma))^\circ$. In particular, since $\gamma$ is consistent, there exists $d\in v\mathcal{C}$ such that $\gamma(d)$ is a bimorphism. By cancellation, we obtain  $(\gamma ^{(1)}(c_\gamma))^\circ= (\gamma ^{(2)}(c_\gamma))^\circ$. Further, since $c_\gamma=c_\delta$, we have $\delta\ast(\gamma ^{(1)}(c_\delta))^\circ= \delta\ast (\gamma ^{(2)}(c_\delta))^\circ$, i.e., $\delta\cdot\gamma ^{(1)}=\delta\cdot\gamma ^{(2)}$. Thus by Lemma \ref{lemconid}, we obtain $\delta\mathrel{\mathscr{L}}^*\gamma$. Hence the lemma.
\end{proof}

\begin{lem}\label{lemic}
$\widehat{\mathcal{C}}$ is idempotent-connected.
\end{lem}
\begin{proof}
Let $\gamma=\epsilon\ast u$ be an arbitrary element in 	the semigroup $\widehat{\mathcal{C}}$. By Lemma \ref{lemab}, there exist idempotents $\epsilon,\delta$ such that $\epsilon\mathrel{\mathscr{R}}^*\gamma\mathrel{\mathscr{L}}^*\delta$. Then since $u$ is a bimorphism, the axiom $(CC\: 4)$ implies that $u$ is consistent. Hence $T^u\colon \sigma (c_\epsilon)\to\sigma(c_\delta)$ has a unique  extension $T\colon \langle c_\epsilon \rangle\to  \langle c_\delta \rangle$ which is an isomorphism such that for each $c_{\epsilon^i}\subseteq c_{\epsilon}$, the map $c_{\epsilon^i}\mapsto T(c_{\epsilon^i})$  is a natural isomorphism. In particular, the inclusions $j_{c_{\epsilon^i}}^{c_\epsilon}$ and $j_{T(c_{\epsilon^i})}^{c_\delta}$ split in $\langle c_\epsilon\rangle$ and $\langle c_\delta\rangle$, respectively. This is illustrated by the following commutative diagrams in the categories $\sigma(\mathcal{C})$ and $\langle\mathcal{C}\rangle$, respectively.
\begin{equation*}\label{}
\xymatrixcolsep{2pc}\xymatrixrowsep{3pc}\xymatrix
{
	c_\epsilon \ar[rr]^{u}   	&& c_\delta  &&c_\epsilon\ar@{->}[d]_{h_i}\ar@{->}[rr]^{T(c_{\epsilon})}&&c_\delta\ar@{->}[d]^{k_i}\\       
	c_{\epsilon^i}\ar[u]^{j_{c_{\epsilon^i}}^{c_{\epsilon}}}\ar[rr]^{(j_{c_{\epsilon^i}}^{c_{\epsilon}}u)^\circ} && T^u(c_{\epsilon^i})\ar[u]_{j_{T^u(c_{\epsilon^i})}^{c_\delta}}&&c_{\epsilon^i}\ar@{->}[rr]^{T(c_{\epsilon^i})}&&T(c_{\epsilon^i}) 
}
\end{equation*}
Now, to show that $\widehat{\mathcal{C}}$ is idempotent-connected, by Lemma \ref{lemico}, it suffices to build a bijection $\beta\colon \omega(\epsilon)\to \omega(\delta)$ satisfying $\epsilon^i\gamma=\gamma(\epsilon^i\beta)$ for all $\epsilon^i\in \omega(\epsilon)$. So, define a function $\beta\colon\omega(\epsilon)\to \omega(\delta)$ as follows:
\begin{equation}
\beta\colon \epsilon^i\mapsto\delta\ast k_i \text{ for all }\epsilon^i\in \omega(\epsilon)
\end{equation}  
where $k_i$ is the retraction in $\langle c_\delta \rangle$ such that $j_{T(c_{\epsilon^i})}^{c_\delta} k_i= 1_{T(c_{\epsilon^i})}$. Since $T$ is an isomorphism, $\beta$ is well-defined and is a bijection. Now, for $\epsilon^i\in \omega(\epsilon)$,  
\begin{align*}
\epsilon^i\cdot\gamma &= \epsilon^i\cdot(\epsilon\ast u)\\  
&= \epsilon^i\ast(\epsilon(c_{\epsilon^i})u)^\circ&(\text{by }(\ref{eqnbin}))\\
&= (\epsilon\cdot\epsilon^i)\ast(j_{c_{\epsilon^i}}^{c_\epsilon}u)^\circ &(\text{since }\epsilon^i\mathrel{\omega}\epsilon)\\ 
&=\epsilon\ast\epsilon^i(c_\epsilon) (j_{c_{\epsilon^i}}^{c_\epsilon}u)^\circ&(\text{by }(\ref{eqnbin}))\\
&=\epsilon\ast h_i (j_{c_{\epsilon^i}}^{c_\epsilon}u)^\circ\\
&=\epsilon\ast u k_i & (\text{from the above commutative diagram})\\
&=\gamma\ast k_i.
\end{align*}
Also $$\gamma(\epsilon^i\beta)=\gamma\cdot(\delta\ast k_i)=\gamma\ast(\delta(c_\gamma)k_i)^\circ=\gamma\ast k_i.$$
Hence the lemma.
\end{proof}

\begin{thm}\label{thmconc}
The semigroup $\widehat{\mathcal{C}}$ is concordant.
\end{thm}
\begin{proof}
By Proposition \ref{pronormal}, Lemma \ref{lemab} and Lemma \ref{lemic}, the semigroup $\widehat{\mathcal{C}}$ is an idempotent-connected abundant semigroup whose idempotents generate a regular subsemigroup. Hence the theorem.
\end{proof}

As in \cite{cross}, a functor $F\colon\mathcal{C}\to \mathcal{D}$ is said to be $v$-\emph{surjective}, $v$-\emph{injective} or $v$-\emph{bijective} if the object map $vF$ has the corresponding property. A functor $F\colon\mathcal{C}\to \mathcal{D}$ is said to be an \emph{isomorphism} if it is $v$-bijective and fully-faithful. Two consistent (normal) categories are said to be isomorphic if there is an inclusion preserving isomorphism between them. The following theorem follows from the similar result \cite[Theorem III.19]{cross} for normal categories.

\begin{thm}
Let $\mathcal{C}$ be a consistent category and $\widehat{\mathcal{C}}$ be its associated concordant semigroup of consistent cones. Define $F\colon\mathcal{C}\to \mathbb{L}(\widehat{\mathcal{C}})$ as follows:
\begin{align*}
vF(c)= \widehat{\mathcal{C}}\epsilon \text{ and }
F(f)= \rho(\epsilon,\epsilon\ast f^\circ, \epsilon')
\end{align*}
where $\epsilon,\epsilon'\in E(\widehat{\mathcal{C}})$ such that $c_\epsilon=c, c_{\epsilon'}=d$ and  $f\colon c\to d$. Then $F$ is an isomorphism of consistent categories.
\end{thm}
The above theorem and the discussion in Section \ref{secconc} gives the following corollary.
\begin{cor}
A category is consistent if and only if it is isomorphic to the category $\mathbb{L}(S)$ for some concordant semigroup $S$.
\end{cor}
Recall from \cite{clif} that a right regular representation of a semigroup $S$ is a homomorphism $\rho\colon a\mapsto\rho_a$ of $S$ into the full transformation semigroup $\mathscr{T}_S$. Then $\rho\colon S\to S_\rho$ is a surjective homomorphism where $S_\rho$ is the image of $\rho$. The following proposition is a direct generalisation of \cite[Theorem III.16]{cross}.
\begin{pro}\label{prosr}
Let $S$ be a concordant semigroup. Then the map $a\mapsto\rho^a$ (where $\rho^a$ is the principal cone determined by $a$) defines a homomorphism $\tilde{\rho}\colon S\to \widehat{\mathbb{L}(S)}$. Also the map $\rho_a\mapsto\rho^a$ defines an injective homomorphism $\phi\colon S_\rho\to \widehat{\mathbb{L}(S)}$ such that the diagram below commutes.
\begin{equation*}\label{}
\xymatrixcolsep{2pc}\xymatrixrowsep{3pc}\xymatrix
{
	&S \ar[rd]^{\tilde{\rho}}\ar[ld]_{{\rho}} 	& \\       
	S_\rho\ar[rr]_{{\phi}}&& \widehat{\mathbb{L}(S)} }
\end{equation*}
In particular $S$ is isomorphic to a subsemigroup of $\widehat{\mathbb{L}(S)}$ via $\tilde{\rho}$ if and only if $\rho$ is injective.
\end{pro}

\begin{rmk}\label{rmkdual}
Dually, we can define the consistent category $\mathbb{R}(S)$ of principal right ideals generated by the idempotents of a concordant semigroup $S$ as follows:
\begin{eqnarray*}
	v\mathbb{R}(S) =& \{ eS : e \in E(S) \}\\
	\mathbb{R}(S)(eS,fS) =& \{ \lambda(e,u,f) : u\in fSe \}.
\end{eqnarray*}
It can be easily shown that the dual properties regarding the category $\mathbb{L}(S)$ hold for the category $\mathbb{R}(S)$.
\end{rmk}

\section{Consistent dual and cross-connections of a concordant semigroup}\label{secdual}

We have seen in the previous sections that given a concordant semigroup $S$, the categories $\mathbb{L}(S)$ and $\mathbb{R}(S)$ are consistent categories. So a natural converse question arises: given two consistent categories $\mathcal{C}$ and $\mathcal{D}$, under what conditions can we assert the existence of a concordant semigroup $S$ such that $\mathcal{C}$ and $\mathcal{D}$ are isomorphic to $\mathbb{L}(S)$ and $\mathbb{R}(S)$, respectively. To answer this question, we first need to understand the relationship the consistent categories $\mathbb{L}(S)$ and $\mathbb{R}(S)$. This relationship will be described in this section using a pair of functors $\Gamma_S$ and $\Delta_S$, which shall be called a {cross-connection}. 

To this end, we need to introduce the notion of a {dual category} associated with a consistent category. This will generalise the notion of a normal dual of a normal category \cite{cross} and also help us characterise the consistent category $\mathbb{R}(\widehat{\mathcal{C}})$ associated with the concordant semigroup $\widehat{\mathcal{C}}$. Recall that (see \cite{mac}) given a category $\mathcal{C}$, the class of all functors from $\mathcal{C}$ to the category $\mathbf{Set}$ with natural transformations as morphisms forms a category $[\mathcal{C},\mathbf{Set}]$. 

The {consistent dual} $\mathcal{C}^*$ of a consistent category $\mathcal{C}$ is defined as a subcategory of the category $[\mathcal{C},\mathbf{Set}] $ such that the objects of $\mathcal{C}^*$ are certain special set-valued functors called $H$-functors.

\subsection{$H$-functor}
Let $\epsilon$ be an idempotent consistent cone in a consistent category. Then for each $c\in v\mathcal{C}$ and $g\colon c\to c'$, we define an $H$-functor $H(\epsilon;-)\colon \mathcal{C}\to \mathbf{Set}$ as follows:
\begin{equation} \label{eqnH}
	\begin{split}
	H({\epsilon};{c})&= \{\epsilon\ast f^\circ : f \in \mathcal{C}(c_{\epsilon},c)\} \text{ and }\\
	H({\epsilon};{g}) \colon H({\epsilon};{c}) &\to H({\epsilon};{d}) \text{ given by }\epsilon\ast f^\circ \mapsto \epsilon\ast (fg)^\circ
	\end{split}
\end{equation}
It can be shown (as in \cite[Lemma III.6]{cross}) that  given an $H$-functor $H(\epsilon;-)$ in a consistent category $\mathcal{C}$, for every pair $(d,\delta)$ such that $\delta\in H(\epsilon;d)$, there exists a unique morphism $f\colon c_\epsilon\to d$ such that $H(\epsilon;f)\colon\epsilon\mapsto\delta$.
Hence the consistent cone $\epsilon$ (or the pair $(c_\epsilon,\epsilon)$, to be precise) will be a universal element for the functor $H(\epsilon;-)$ in $H(\epsilon;c_\epsilon)$. This implies that the functor $H(\epsilon;-)$ is a representable functor, i.e., there exists a natural isomorphism $\eta_\epsilon\colon H(\epsilon;-) \to \mathcal{C}(c_\epsilon,-)$ where $\mathcal{C}(c_\epsilon,-)$ is the covariant hom-functor determined by the object $c_\epsilon$. Observe that the natural isomorphism may be explicitly defined by $\eta_\epsilon\colon c\mapsto (\epsilon\ast f^\circ \mapsto f^\circ)$ for an arbitrary object $c\in v\mathcal{C}$. Now, using Lemma \ref{lemgc} and \cite[Proposition III.7]{cross}, we have the following proposition.
\begin{pro}\label{progreen}
Let $\epsilon,\epsilon'$ be idempotent consistent cones in the semigroup $\widehat{\mathcal{C}}$. Then 
\begin{enumerate}
	\item $\epsilon\mathrel{\mathscr{L}}\epsilon'$ if and only if $c_\epsilon=c_{\epsilon'}$.
	\item $\epsilon\mathrel{\mathscr{R}}\epsilon'$ if and only if $H(\epsilon;-)=H({\epsilon'};-)$.
\end{enumerate}
\end{pro}

Hence given a consistent category $\mathcal{C}$, we define the \emph{consistent dual} $\mathcal{C}^*$ (often referred to as just \emph{dual} in the sequel) as the full subcategory of $[\mathcal{C},\mathbf{Set}] $ such that
$$v\mathcal{C}^*=\{H(\epsilon;-):\epsilon\in E(\widehat{\mathcal{C}})\}.$$
Then  $\mathcal{C}^*$ is a category with subobjects in which the inclusion relation among the functors is defined as follows. Let $H(\epsilon;-),H(\epsilon';-)\colon\mathcal{C} \to \mathbf{Set}$, we say $H(\epsilon;-)$ is a sub functor of $H(\epsilon';-)$ (and write $H(\epsilon;-)\subseteq H(\epsilon';-)$) if for all $c\in v\mathcal{C}$, the sets $H(\epsilon;c)\subseteq H(\epsilon';c)$ and the map $c\mapsto j_{H(\epsilon;c)}^{H(\epsilon';c)}$ is a natural transformation from $H(\epsilon;-)$ to $H(\epsilon';-)$. 

Note that for $\epsilon,\epsilon'\in E(\widehat{\mathcal{C}})$ and $\gamma\in \epsilon'\widehat{\mathcal{C}}\epsilon$, exactly as in \cite[Lemma III.22]{cross}, the map $\lambda(\epsilon,\gamma,\epsilon')\mapsto\widetilde{\gamma}$ where 
\begin{equation}\label{eqngam}
\widetilde{\gamma}= \gamma(c_{\epsilon'})j_{c_\gamma}^{c_\epsilon}
\end{equation}
 is a bijection from $\mathbb{R}(\widehat{\mathcal{C}})(\epsilon\widehat{\mathcal{C}},\epsilon'\widehat{\mathcal{C}})$ onto $\mathcal{C}(c_{\epsilon'},c_\epsilon)$. Then the following theorem is a straightforward generalisation of \cite[Theorem III.25]{cross}.
\begin{thm}\label{thmdualrs}
Let $\mathcal{C}$ be a consistent category. Define $G\colon \mathbb{R}(\widehat{\mathcal{C}}) \to \mathcal{C}^*$ as follows: 
$$vG (\epsilon\widehat{\mathcal{C}}) = H(\epsilon;-)\text{ for each }\epsilon\in E(\widehat{\mathcal{C}})$$ 
and for each $\lambda=\lambda(\epsilon,\gamma,\epsilon')\colon \epsilon\widehat{\mathcal{C}}\to\epsilon'\widehat{\mathcal{C}}$, let $G(\lambda)$ be the natural transformation between the functors $H(\epsilon;-)$ and $H(\epsilon';-)$ making the following diagram commutative where $\widetilde{\gamma}$ is defined by (\ref{eqngam}).
\begin{equation*}\label{}
\xymatrixcolsep{3pc}\xymatrixrowsep{4pc}\xymatrix
{
	H(\epsilon;-) \ar[d]_{G(\lambda)}\ar[rr]^{\eta_\epsilon}   	&& \mathcal{C}(c_{\epsilon},-)\ar[d]^{\mathcal{C}(\widetilde{\gamma},-)}  \\       
	H(\epsilon';-) \ar[rr]^{\eta_{\epsilon'}} && \mathcal{C}(c_{\epsilon'},-)
}
\end{equation*}
Then the functor $G\colon \mathbb{R}(\widehat{\mathcal{C}}) \to \mathcal{C}^*$ is an isomorphism of consistent categories.
\end{thm}
Since $\widehat{\mathcal{C}}$ is a concordant semigroup, using Remark \ref{rmkdual}, the category $\mathbb{R}(\widehat{\mathcal{C}})$ is a consistent category. So, we have the following corollary.
\begin{cor}
Given a consistent category $\mathcal{C}$, its consistent dual $\mathcal{C}^*$ is also a consistent category.
\end{cor} 

\subsection{Cross-connections}
Now, we proceed to describe how the consistent categories $\mathbb{L}(S)$ and $\mathbb{R}(S)$ arising from a concordant semigroup $S$ are interrelated. To that end, first consider the following functor $FS_\rho\colon\mathbb{R}(S)\to \mathbb{R}(\widehat{\mathbb{L}(S)})$. For each $eS\in v\mathbb{R}(S)$ and for each morphism $\lambda(e,u,f)\in \mathbb{R}(S)$,
\begin{equation} \label{eqnfsr}
vFS_\rho(eS) = \rho^e(\widehat{\mathbb{L}(S)}) \quad\text{ and }\quad FS_\rho(\lambda(e,u,f)) = \lambda(\rho^e,\rho^u,\rho^f).
\end{equation}

Exactly as shown in \cite[Proposition IV.1]{cross}, we can prove that $FS_\rho$ is a well defined covariant functor which is inclusion preserving, fully-faithful and for each $eS\in v\mathbb{R}(S)$, the restriction functor $FS_{\rho|(eS)}$ to the ideal $(eS)$ in $\mathbb{R}(S)$ is an isomorphism. This motivates us to define the following notion which will be very crucial in the sequel.

\begin{dfn}\label{dfnlociso}
A functor $F$ between two consistent categories $\mathcal{C}$ and $\mathcal{D}$ is said to be a \emph{local isomorphism} if $F$ is inclusion preserving, fully faithful and for each $c \in v\mathcal{C}$, $F_{|(c)}$ is an isomorphism of the ideal $( c )$ onto $ (F(c)) $.
\end{dfn}

Dually as defined in (\ref{eqnfsr}), we can define another functor $FS_\lambda\colon \mathbb{L}(S)\to \mathbb{R}(\widehat{\mathbb{R}(S)})$ as follows. For each $Se\in v\mathbb{L}(S)$ and for each morphism $\rho(e,u,f)\in \mathbb{L}(S)$,
\begin{equation} \label{eqnfsl}
vFS_\lambda(Se) = \lambda^e(\widehat{\mathbb{R}(S)}) \quad\text{ and }\quad FS_\lambda(\rho(e,u,f)) = \lambda(\lambda^e,\lambda^u,\lambda^f).
\end{equation}
Summarising the above discussion, we have the following proposition which generalises \cite[Proposition IV.1]{cross}.
\begin{pro}
The functors $FS_\rho\colon\mathbb{R}(S)\to \mathbb{R}(\widehat{\mathbb{L}(S)})$ and $FS_\lambda\colon \mathbb{L}(S)\to \mathbb{R}(\widehat{\mathbb{R}(S)})$ as defined in (\ref{eqnfsr}) and (\ref{eqnfsl}) respectively, are local isomorphisms.
\end{pro}
\begin{rmk}
Observe that the local isomorphisms  $FS_\rho$ and $FS_\lambda$  arise from the homomorphism $\tilde{\rho}\colon S \to\widehat{\mathbb{L}(S)}$ (see Proposition \ref{prosr}) and its dual anti-homomorphism $\tilde{\lambda}\colon S \to\widehat{\mathbb{R}(S)}$, respectively. 
\end{rmk}

Now, given a concordant semigroup $S$, we define a pair of functors $\Gamma_S$ and $\Delta_S$ as follows. The functor $\Gamma_S\colon  \mathbb{R}(S) \to \mathbb{L}(S)^*$ is given by
\begin{equation} \label{eqngs}
v\Gamma_S(eS) = H(\rho^e;-) \quad\text{ and }\quad \Gamma_S(\lambda(e,u,f)) = \eta_{\rho^e}\mathbb{L}(S)(\rho(f,u,e),-)\eta_{\rho^f}^{-1}
\end{equation}
and the functor $\Delta_S\colon  \mathbb{L}(S) \to \mathbb{R}(S)^*$ is defined as follows:
\begin{equation}\label{eqnds}
v\Delta_S(Se) = H(\lambda^e;-) \quad\text{ and }\quad \Delta_S(\rho(e,u,f)) = \eta_{\lambda^e}\mathbb{R}(S)(\lambda(f,u,e),-)\eta_{\lambda^f}^{-1}.
\end{equation}

First observe that by Theorem \ref{thmdualrs}, the category $\mathbb{R}(\widehat{\mathbb{L}(S)})$ is isomorphic to ${\mathbb{L}(S)}^*$ as consistent categories, via the functor say $\overrightarrow{G}$. Similarly, the category $\mathbb{R}(\widehat{\mathbb{R}(S)})$ is isomorphic to ${\mathbb{R}(S)}^*$ via the functor, say $\overleftarrow{G}$. Comparing the functors $\Gamma_S$ and $\Delta_S$ with definitions in (\ref{eqnfsr}), (\ref{eqnfsl}) and Theorem \ref{thmdualrs}, we see that 
$$\Gamma_S=FS_\rho\circ\overrightarrow{G} \quad\text{ and }\quad \Delta_S=FS_\lambda\circ\overleftarrow{G}.$$
Now, since the functor $FS_\rho\colon\mathbb{R}(S)\to \mathbb{R}(\widehat{\mathbb{L}(S)})$ is a local isomorphism and the functor $\overrightarrow{G}\colon \mathbb{R}(\widehat{\mathbb{L}(S)})\to {\mathbb{L}(S)}^*$ is an isomorphism, the functor $\Gamma_S\colon\mathbb{R}(S)\to {\mathbb{L}(S)}^*$ is a local isomorphism. Arguing similarly for the functor $\Delta_S$, we have the following theorem: 
\begin{thm}\label{thmlociso}
The functor $\Gamma_S\colon\mathbb{R}(S)\to {\mathbb{L}(S)}^*$ and the functor $\Delta_S\colon\mathbb{L}(S)\to {\mathbb{R}(S)}^*$ as defined in (\ref{eqngs}) and (\ref{eqnds}) respectively, are local isomorphisms. 
\end{thm}
Given an $\epsilon\in E(\widehat{\mathcal{C}})$, we define as follows, the set $MH(\epsilon;-)$, called as the M-set of the idempotent cone $\epsilon$ and denoted in the sequel by just $M\epsilon$.
\begin{equation}\label{eqnms}
MH(\epsilon;-)= M\epsilon = \{ c\in v\mathcal{C}: \epsilon(c) \text{ is an isomorphism} \}.
\end{equation}

Further, observe the following interrelationship of the functors $\Gamma_S$ and $\Delta_S$. For objects $Se\in v\mathbb{L}(S)$ and $eS\in v\mathbb{R}(S)$,
\begin{equation}
Se \in M\Gamma_S(eS) \text{ if and only if } eS\in M\Delta_S(Se).
\end{equation}
The above discussion leads us to the definition of a cross-connection.
\begin{dfn} \label{ccxn}
Let $\mathcal{C}$ and $\mathcal{D}$ be consistent categories. A \emph{cross-connection} between $\mathcal{C}$ and $\mathcal{D}$ is a quadruplet $(\mathcal{C},\mathcal{D};{\Gamma},\Delta)$ where $\Gamma\colon  \mathcal{D} \to \mathcal{C}^*$ and $\Delta\colon  \mathcal{C} \to \mathcal{D}^*$ are local isomorphisms such that for $c \in v\mathcal{C}$ and $d \in v\mathcal{D}$ 
\begin{equation}\label{eqncxnms}
c \in M\Gamma(d) \iff d\in M\Delta(c).
\end{equation}
\end{dfn}

\begin{rmk}
Observe that we define a cross-connection using two functors emulating Grillet's \cite{gril1} original definition using two maps, unlike in \cite{cross0,cross,romeo} where a cross-connection is defined using a single functor. One can easily observe that our definition is equivalent to the definition using a single functor and as shown in \cite{cross}, the second functor is uniquely determined by the first. But our formulation although being less economical, will help us recover the semigroup from a cross-connection in a much easier manner (see next section).
\end{rmk}

Summarising the above discussion, we have proved the following theorem.
\begin{thm}\label{thmcxns}
Let $S$ be a concordant semigroup with consistent categories $\mathbb{L}(S)$ and $\mathbb{R}(S)$. Define functors $\Gamma_S$ and $\Delta_S$ as in (\ref{eqngs}) and (\ref{eqnds}). Then $\Omega S= (\mathbb{L}(S),\mathbb{R}(S);\Gamma_S,\Delta_S)$ is a cross-connection between  $\mathbb{L}(S)$ and $\mathbb{R}(S)$.
\end{thm}

\section{Concordant semigroup of a cross-connection}\label{seccxn}

In the previous section, we showed how a concordant semigroup gives rise to a cross-connection. In this section, we describe the converse: the concordant semigroup arising from a cross-connection between two consistent categories. Recall from Section \ref{seccons} that given a consistent category, we have an associated concordant semigroup. Naturally, we shall be identifying the concordant semigroup associated with a cross-connection as a subdirect product of the concordant semigroups arising from the two consistent categories, i.e., as a semigroup of ordered pairs of consistent cones which `respect' the cross-connection. But for this, we need a deeper analysis of the cross-connection functors and their interrelationship.

\subsection{The idempotent cones $\gamma(c,d)$ and $\delta(c,d)$}
First, observe that for small categories $\mathcal{C}$, $\mathcal{D}$ and the category $\mathbf{Set}$, we have the following isomorphism \cite{mac}:
$$[\mathcal{C},[\mathcal{D},\mathbf{Set}]]\cong [\mathcal{C}\times\mathcal{D},\mathbf{Set}].$$

This implies that any functor from $\mathcal{C}$ to $\mathcal{D}^*$ (or from $\mathcal{D}$ to $\mathcal{C}^*$) will uniquely determine a bifunctor from $\mathcal{C}\times\mathcal{D}$ to $\mathbf{Set}$. 

Hence, given a cross-connection $\Omega=(\mathcal{C},\mathcal{D};{\Gamma},\Delta)$, it gives rise to two bifunctors $\Gamma(-,-)$ and $\Delta(-,-)$ from $\mathcal{C}\times\mathcal{D}$ to $\mathbf{Set}$ defined as follows. For all $(c,d) \in v\mathcal{C}\times v\mathcal{D}$ and $(f,g)\colon (c,d) \to (c',d')$,
\begin{align}\label{eqnbif}
\begin{split}
v\Gamma(c,d) = \Gamma(d)(c)	\text{, }& \Gamma(f,g) = \Gamma(g)(c)\Gamma(d')(f) = \Gamma(d)(f)\Gamma(g)(c'); \\
v\Delta(c,d) = \Delta(c)(d)\text{ and }& \Delta(f,g) = \Delta(c)(g)\Delta(f)(d') = \Delta(f)(d)\Delta(c')(g).
\end{split}
\end{align}

Now, given a cross-connection $\Omega=(\mathcal{C},\mathcal{D};{\Gamma},\Delta)$, define a set:
\begin{equation}\label{eqneo}
E_\Omega=\{ (c,d)\in v\mathcal{C}\times v\mathcal{D} : c\in M\Gamma(d)\}
\end{equation}
We shall show later that the above defined set is in fact the regular biordered set associated with the cross-connection $\Omega$. As a beginning, we identify the idempotent cones associated with an element $(c,d)\in E_\Omega$. For that, we gather the following lemma from \cite{cross}.
\begin{lem}\label{lemmset}
	Let $\epsilon$ be an idempotent cone in a consistent category ${\mathcal{C}}$. Then $c\in MH(\epsilon;-)$ if and only if there exists a unique idempotent cone $\xi$ in ${\mathcal{C}}$ such that $H(\xi;-)=H(\epsilon;-)$ and $c_\xi=c$.
\end{lem}
\begin{proof}
	Given an idempotent cone $\xi$ in ${\mathcal{C}}$ such that $H(\xi;-)=H(\epsilon;-)$ and $c_\xi=c$, by Proposition \ref{progreen}, we have $\xi\mathrel{\mathscr{R}}\epsilon$ and using Lemma \ref{lemgc} there exist unique epimorphisms $\epsilon(c_\xi)=h\colon c_\xi \to c_\epsilon$ and $\xi(c_\epsilon)=k\colon c_\epsilon \to c_\xi$ such that
	$$\epsilon=\xi\ast h \text{ and } \xi=\epsilon\ast k.$$ 
	Since $\epsilon(c_\epsilon)=1_{c_\epsilon}$ and $\xi_{c_\xi}=1_{c_\xi}$, we have $hk=1_{c_\xi}$ and $kh=1_{c_\epsilon}$ so that $h=\epsilon(c_\xi)$ is an isomorphism. Hence $c=c_\xi \in MH(\epsilon;-)$.
	
	Conversely, if $\epsilon(c)$ is an isomorphism, say $u$, then define $\xi=\epsilon\ast u^{-1}$. Now, 
	$$\xi(c)=\epsilon\ast u^{-1}(c)=\epsilon(c) u^{-1}= u u^{-1} = 1_c$$ 
	and so $\xi$ is an idempotent cone with apex $c$. Using Lemma \ref{lemgc} and Proposition \ref{progreen}, we see that $H(\xi;-)=H(\epsilon;-)$ and hence the lemma. 
\end{proof}

Suppose $(c,d)\in E_\Omega$ and $\Gamma(d)=H(\epsilon;-)$ for some $\epsilon\in \widehat{\mathcal{C}}$ so that $c\in MH(\epsilon;-)$. Then by Lemma \ref{lemmset}, there is a uniquely defined idempotent cone $\xi$ in $\mathcal{C}$ such that 
\begin{equation}\label{eqngcd}
c_\xi=c\text{ and }H(\xi;-)=H(\epsilon;-)=\Gamma(d).
\end{equation}

We shall denote this idempotent cone by $\gamma(c,d)$ in the sequel. Similarly, for each pair $(c,d)\in E_\Omega$, there is a unique idempotent cone $\delta(c,d)\in \widehat{\mathcal{D}}$ such that
\begin{equation}
c_{\delta(c,d)}=c\text{ and }H(\delta(c,d);-)=\Delta(c).
\end{equation}

\subsection{Transpose}
Observe that if $c'\in M\Gamma(d)$, then $(c',d)\in E_\Omega$. Then for the idempotent cone $\delta(c',d)$ in the category $\mathcal{D}$, since $H(\delta(c',d);-)$ is a representable functor, there is a natural isomorphism $\eta_{\delta(c',d)}\colon\Delta(c')\to \mathcal{D}(d,-)$. Similarly, we have a natural isomorphism $\eta_{\delta(c,d')}\colon \Delta(c)\to \mathcal{D}(d',-)$. So for a morphism $f\colon c' \to c$ in $\mathcal{C}$, we see that $\mu=\eta_{\delta(c',d)}^{-1}\circ \Delta(f) \circ \eta_{\delta(c,d')}$ is a natural transformation from $\mathcal{D}(d,-)$ to $\mathcal{D}(d',-)$. Now, using Yoneda Lemma \cite{mac}, there is a unique morphism from $g\colon d' \to d$ in the category $\mathcal{D}$ such that $\mu=\mathcal{D}(g,-)$, as shown in the following commutative diagram.
\begin{equation*}\label{}
\xymatrixcolsep{3pc}\xymatrixrowsep{4pc}\xymatrix
{
	c'\ar[d]_{f}&\Delta(c') \ar[d]_{\Delta(f)}\ar[rr]^{\eta_{\delta(c',d)}}   	&& \mathcal{D}(d,-)\ar[d]^{\mathcal{D}(g,-)}&d  \\       
	c& \Delta(c) \ar[rr]^{\eta_{\delta(c,d')}} && \mathcal{D}(d',-)&d'\ar[u]_{g}
}
\end{equation*}

Then the unique morphism $g\in \mathcal{D}(d',d)$ is known as the \emph{transpose} of $f\in\mathcal{C}(c',c)$ and will be denoted by $f^\ddagger$ in the sequel. 

\begin{rmk}
Observe that the transpose of a given morphism $f\in\mathcal{C}(c',c)$ is not unique, in general. For each $d'\in M\Delta(c)$ and $d\in M\Delta(c')$, there is a unique transpose $f^\ddagger\in\mathcal{D}(d',d)$.
\end{rmk}
\begin{rmk}
Dually, given a morphism $g\in \mathcal{D}(d',d)$, we can define the tranpose $g^\ddagger$ as the unique morphism in $\mathcal{C}(c',c)$ such that $\mathcal{C}(g^\ddagger,-)=\eta_{\gamma(c,d')}^{-1}\circ \Gamma(g) \circ \eta_{\gamma(c',d)}$.
\end{rmk}

Using the above notations and the definitions of the $H$-functor in (\ref{eqnH}) and of the bifunctors $\Gamma(-,-)$ and $\Delta(-,-)$ in (\ref{eqnbif}), we can see that 
\begin{equation*}
\begin{split}
\Gamma(c,d)=\{\gamma(c',d)\ast f^\circ : c' \in M\Gamma(d) \text{ and } f\in \mathcal{C}(c',c) \}, \\
\Delta(c,d)=\{\delta(c,d')\ast g^\circ : d' \in M\Delta(c) \text{ and } g\in \mathcal{D}(d',d) \}.
\end{split}
\end{equation*}

Then we have the following theorem which is an exact generalisation of \cite[Theorem IV.16]{cross}, in the notation introduced above.
\begin{thm}\label{thmnatiso}
Given a cross-connection $\Omega=(\mathcal{C},\mathcal{D};{\Gamma},\Delta)$ with bifunctors $\Gamma(-,-)$ and $\Delta(-,-)$, for each $(c,d)\in v\mathcal{C}\times v\mathcal{D}$, the map $\chi(c,d)\colon\Gamma(c,d)\to \Delta(c,d)$ given by
\begin{equation*}
\chi(c,d)\colon \gamma(c',d)\ast f^\circ \mapsto \delta(c,d')\ast g^\circ
\end{equation*}
is a bijection, where  $c' \in M\Gamma(d) \text{ and }  d' \in M\Delta(c)$ and $g\in \mathcal{D}(d',d)$ is the transpose of the morphism $f\in \mathcal{C}(c',c)$. Also the map $(c,d)\mapsto\chi(c,d)$ defines a natural isomorphism between the bifunctors $\Gamma(-,-)$ and $\Delta(-,-)$.
\end{thm}

\subsection{Linked pairs}
Now, consider the following subsets of the semigroups $\widehat{\mathcal{C}}$ and $\widehat{\mathcal{D}}$ we obtain from the bifunctors $\Gamma(-,-)$ and $\Delta(-,-)$.
\begin{subequations}\label{eqnug}
	\begin{align}
	\widehat{\Gamma} = & \bigcup\: \{  \Gamma(c,d) : (c,d) \in v\mathcal{C} \times v\mathcal{D} \} \\
	\widehat{\Delta} = & \bigcup\: \{  \Delta(c,d) : (c,d) \in v\mathcal{C} \times v\mathcal{D} \}
	\end{align}
\end{subequations}

One can easily verify that $\widehat{\Gamma}$ and $\widehat{\Delta}$ are subsemigroups of $\widehat{\mathcal{C}}$ and $\widehat{\mathcal{D}}$, respectively.  We proceed to show that $\widehat{\Gamma}$ and $\widehat{\Delta}$ are in fact concordant semigroups. We begin with the following lemma.

\begin{lem}\label{lemug}
Let $\Omega=(\Gamma,\Delta;\mathcal{C},\mathcal{D})$ be a cross-connection. Then a cone $\gamma\in \widehat{\Gamma}$ if and only if $\gamma=\gamma(c_1,d_1)\ast u$, where $u$ is bimorphism in $\mathcal{C}$ and $(c_1,d_1)\in E_\Omega$.
\end{lem}
\begin{proof}
Suppose $\gamma\in \widehat{\Gamma}$ so that $\gamma=\gamma(c',d)\ast f^\circ$ where $f\colon c' \to c$. Let $f^\circ=eu$ be the consistent factorisation of $f^\circ$, so that $u$ is a bimorphism. Then $\gamma=\gamma(c',d)\ast eu=(\gamma(c',d)\ast e)\ast u$. Let $\epsilon= \gamma(c',d)\ast e$, then since $e$ is a retraction, $\epsilon$ is an idempotent and $c_1=c_\epsilon\subseteq c'$. Also by \cite[Proposition III.9]{cross}, we have $H(\epsilon;-)\subseteq H(\gamma(c',d);-)=\Gamma(d)$. Since $\Gamma$ is a local isomorphism, there exists a unique $d_1\subseteq d$ such that $\Gamma(d_1)
=H(\epsilon;-)$. So $\epsilon=\gamma(c_1,d_1)$ and hence $\gamma=\gamma(c_1,d_1)\ast u$. The converse is clear.
\end{proof}
Dually, we can prove that a cone $\delta\in \widehat{\Delta}$ if and only if $\delta=\delta(c_1,d_1)\ast u$ where $u$ is bimorphism in $\mathcal{D}$. 
\begin{lem}\label{lemeug}
The set of idempotents $E(\widehat{\Gamma})$ is given by 
$$E(\widehat{\Gamma})=\{ \gamma(c,d) : (c,d) \in v\mathcal{C} \times v\mathcal{D}\}.$$
\end{lem}
\begin{proof}
If $\gamma=\gamma(c,d)$, then clearly $\gamma\in E(\widehat{\Gamma})$. Conversely, let $\gamma\in E(\widehat{\Gamma})$, then by the above lemma $\gamma=\gamma(c',d)\ast u$, for a bimorphism $u\colon c'\to c$. Now, since $\gamma$ is an idempotent cone, $\gamma(c)=\gamma(c',d)\ast u \ (c)=\gamma(c',d)(c) \ u =1_c$. Also, $$1_{c'} \ u = u = u\ 1_c = u (\gamma(c',d)(c) \ u)=(u \gamma(c',d)(c))u.$$
From cancellation, we get $1_{c'} =u \gamma(c',d)(c)$. Therefore $\gamma(c',d)(c)$ is an isomorphism with the morphism $u$ as the inverse. So, $c\in M\Gamma(d)=H(\gamma(c',d);-)$ and $\gamma\mathrel{\mathscr{R}}\gamma(c',d)$. So by Proposition \ref{progreen}, the functor $H(\gamma;-)= H(\gamma(c',d);-)$. Now, using the uniqueness in Lemma \ref{lemmset} and (\ref{eqngcd}), we have $\gamma=\gamma(c,d)$.
\end{proof}

To show that $\widehat{\Gamma}$ is concordant, we need to first show that the regular elements in $\widehat{\Gamma}$ form a subsemigroup. Or equivalently as in Proposition \ref{pronormal},  we need to identify a full regular subsemigroup of $\widehat{\Gamma}$ such that their biordered sets are isomorphic.

Recall that every idempotent cone in a consistent category is normal and the cross-connection definition depends only on the idempotents. Also observe that the inclusion functor $J(\overline{\mathcal{C}},\mathcal{C})$ is v-surjective for any consistent category $\mathcal{C}$ and its corresponding normal category $\overline{\mathcal{C}}$ as defined in Lemma \ref{lemnorm}. Further, the biorder quasi orders in the sets $E(\widehat{\overline{\mathcal{C}}})$ and $E(\widehat{\mathcal{C}})$ of idempotents, coincide by the discussion following Lemma \ref{lemgc}.  Hence the following lemma can be easily verified.

\begin{lem}\label{lemugrb}
Let $\Omega=(\Gamma,\Delta;\mathcal{C},\mathcal{D})$ be a cross-connection between consistent categories $\mathcal{C}$ and $\mathcal{D}$. If $\overline{\mathcal{C}}$
and $\overline{\mathcal{D}}$ be normal categories as defined in Lemma \ref{lemnorm}, then $\overline{\Omega}=(\overline{\Gamma},\overline{\Delta};\overline{\mathcal{C}},\overline{\mathcal{D}})$ is a cross-connection between normal categories where $\overline{\Gamma}=\Gamma_{|\overline{\mathcal{D}}}$ and $\overline{\Delta}=\Delta_{|\overline{\mathcal{C}}}$. Further, if $\widehat{\Gamma}$ is the semigroup as defined above in (\ref{eqnug}) and the set $\widehat{\overline{\Gamma}}= \bigcup\: \{  \overline{\Gamma}(c,d) : (c,d) \in v\overline{\mathcal{C}} \times v\overline{\mathcal{D}} \}$, then
\begin{equation*}
\widehat{\overline{\Gamma}}=\{ \gamma\in \widehat{\Gamma} : \gamma\text{ is a normal cone in }\mathcal{C} \}
\end{equation*} 
and $\widehat{\overline{\Gamma}}$ is a full regular subsemigroup of $\widehat{\Gamma}$ such that their biordered sets coincide.
\end{lem}
To see the details of the regular semigroup $\widehat{\overline{\Gamma}}$ (denoted in \cite{cross} as $U\Gamma$), please refer to \cite[Section IV.5.1]{cross}. 
\begin{rmk}
The above lemma reflects the fact that the cross-connection definitions of consistent categories of this article, normal categories in \cite{cross} and even regular partially ordered sets in \cite{gril1} are all equivalent. This is because in all the three cases, we are building the same underlying object: a regular biordered set.
\end{rmk}
Now, we proceed to show that the semigroup $\widehat{\Gamma}$ is concordant.
\begin{pro}\label{prog}
Let $\Omega=(\Gamma,\Delta;\mathcal{C},\mathcal{D})$ be a cross-connection and  $\widehat{\Gamma}$  be the semigroup defined in (\ref{eqnug}). Then $\widehat{\Gamma}$ is a concordant semigroup.
\end{pro}
\begin{proof}
By Lemma \ref{lemugrb}, the regular elements in $\widehat{\Gamma}$ form a subsemigroup $\widehat{\overline{\Gamma}}$.

Let $\gamma\in \widehat{\Gamma}$ then $\gamma=\gamma(c,d)\ast u$ where $u\colon c \to c'$ is a bimorphism in $\mathcal{C}$. Now, define $\gamma^\dagger=\gamma(c,d)$ and  $\gamma^*=\gamma(c',d')$ where $d'\in M\Delta(c')$. Then similar to the proof of Lemma \ref{lemab}, we can verify that $\gamma^\dagger\: \mathrel{\mathscr{R}}^*\: \gamma\: \mathrel{\mathscr{L}}^*\: \gamma^*$. Hence $\widehat{\Gamma}$ is an abundant semigroup.  

Finally, since $\mathcal{C}$ is consistent, the bimorphism $u$ is consistent. So, arguing similarly as in the proof of Lemma \ref{lemic}, we can show that $\widehat{\Gamma}$ is idempotent-connected. Hence $\widehat{\Gamma}$ is a concordant semigroup.
\end{proof}

Further, we have the following exact generalisation of \cite[Proposition IV.31]{cross}. 
\begin{pro}
Let $\Omega=(\Gamma,\Delta;\mathcal{C},\mathcal{D})$ be a cross-connection and  $\widehat{\Gamma}$ be the semigroup defined in (\ref{eqnug}) with the set of idempotents $E(\widehat{\Gamma})$ as defined in Lemma \ref{lemeug}. Then $\widehat{F}\colon \mathcal{C}\to \mathbb{L}(\widehat{\Gamma})$ defined by:
\begin{equation*}
v\widehat{F}(c) \ = \widehat{\Gamma}\gamma(c,d) \text{ and } \widehat{F}(f)=\rho(\gamma(c,d), \gamma(c,d)\ast f^\circ,\gamma(c',d'))_{|\widehat{\Gamma}}
\end{equation*}
for all $c\in v\mathcal{C}$ and $f\in \mathcal{C}(c,c')$, is an isomorphism.
\end{pro}

\begin{rmk}
Dually we can show that $\widehat{\Delta}$ is also a concordant semigroup such that $\mathbb{L}(\widehat{\Delta})$ is isomorphic to $\mathcal{D}$.
\end{rmk}

Now, we proceed to build the cross-connection semigroup associated with the cross-connection as a subdirect product of the concordant semigroups $\widehat{\Gamma}$ and $\widehat{\Delta}$. Recall that $\chi$ as defined in Theorem \ref{thmnatiso} is a natural isomorphism between the bifunctors $\Gamma(-,-)$ and $\Delta(-,-)$. This give rise to a `linking' between the concordant semigroups $\widehat{\Gamma}$ and $\widehat{\Delta}$.
\begin{dfn}
Given a cross-connection $\Omega=(\Gamma,\Delta;\mathcal{C},\mathcal{D})$, a consistent cone $\gamma \in \widehat{\Gamma}$ is said to be \emph{linked} to $\delta \in \widehat{\Delta}$ if there is a $(c,d) \in v\mathcal{C} \times v\mathcal{D}$ such that $\gamma \in \Gamma(c,d)$ and $ \delta = \chi(c,d)(\gamma)$; we then say that the pair $(\gamma,\delta)$ is a linked pair.
\end{dfn} 

\subsection{The cross-connection semigroup}
Given a cross-connection $\Omega=(\Gamma,\Delta;\mathcal{C},\mathcal{D})$ of consistent categories $\mathcal{C}$ and $\mathcal{D}$, define the set 
\begin{equation}
\mathbb{S}\Omega=\{ (\gamma,\delta) \in \widehat{\Gamma}\times \widehat{\Delta} : (\gamma,\delta) \text{ is linked }\:\}.
\end{equation}
Define an operation on $\mathbb{S}\Omega$ as follows: 
\begin{equation*}
(\gamma , \delta) \circ ( \gamma' , \delta') = (\gamma \cdot \gamma' , \delta' \cdot \delta) \  \text{  for all  }(\gamma,\delta),( \gamma' , \delta') \in \mathbb{S}\Omega.
\end{equation*} 	
Suppose $(\gamma,\delta),(\gamma',\delta') \in \mathbb{S}\Omega$, then as in the \cite[Lemma IV.30]{cross}, we have $\gamma \cdot \gamma'$ is linked to $ \delta' \cdot \delta$. Hence $\mathbb{S}\Omega$ is a semigroup and it will be called the \emph{cross-connection semigroup}  determined by $\Omega$.
\begin{lem}\label{lemesg}
Let $\Omega=(\Gamma,\Delta;\mathcal{C},\mathcal{D})$ be a cross-connection with $\mathbb{S}\Omega$ as the cross-connection semigroup, then the set of idempotents of $\mathbb{S}\Omega$ is given by:
\begin{equation*}
E(\mathbb{S}\Omega)=\{(\gamma(c,d),\delta(c,d)) : (c,d)\in E_\Omega \}
\end{equation*}
\end{lem}
\begin{proof}
Clearly, since $\gamma(c,d)$ and $\delta(c,d)$ are idempotents and $ \delta(c,d) = \chi_\Gamma(c,d)(\gamma(c,d))$, we see that $(\gamma(c,d),\delta(c,d)) \in E_\Omega$. Conversely, if $(\gamma,\delta)$ is an idempotent, then $\gamma^2=\gamma$. So, by Lemma \ref{lemeug} we have $\gamma\in E(\widehat{\Gamma})$ and hence $\gamma=\gamma(c,d)$ for some $(c,d)\in E_\Omega$. Then $\chi_\Gamma(c,d)(\gamma(c,d))=\delta(c,d)$. Then by the well-definedness of $\chi(c,d)$, we have $\delta=\delta(c,d)$. Hence the lemma.
\end{proof}

Emulating the discussion in \cite[Section V.1.2]{cross}, we can see that the set $E_\Omega$ is bijective with the set $E(\mathbb{S}\Omega)$ under the map 
$$(c,d)\mapsto(\gamma(c,d),\delta(c,d)).$$ 
As outlined in Lemma \ref{lemugrb}, the set of idempotents of $E(\mathbb{S}\Omega)$ and $E(\mathbb{S}\overline{\Omega})$ are equal. Further, as in \cite[Section V.1.2]{cross}, we can show that biorder quasi orders in the set $E(\mathbb{S}\Omega)=E_\Omega$ is given by:
\begin{equation*}
(c,d)\mathrel{\omega}^l(c',d') \iff c\subseteq c' \ \text{ and } \ (c,d)\mathrel{\omega}^r(c',d') \iff d\subseteq d'.
\end{equation*}
Then $E_\Omega$ forms a regular biordered set with the basic products and sandwich sets as described in \cite[Section V.1.2]{cross}.

\begin{thm}\label{thmcxncon}
Given a cross-connection $\Omega=(\Gamma,\Delta;\mathcal{C},\mathcal{D})$ of consistent categories $\mathcal{C}$ and $\mathcal{D}$, the cross-connection semigroup $\mathbb{S}\Omega$ is concordant.
\end{thm}
\begin{proof}
First, as discussed above $E(\mathbb{S}\Omega)$ is a regular biordered set and hence the regular elements in  $\mathbb{S}\Omega$ form a regular subsemigroup.

Now,  given $(\gamma,\delta) \in \mathbb{S}\Omega$, suppose $\gamma=\gamma(c,d)\ast u$ where $u\colon c\to c'$ is a bimorphism, then 
$$\delta=\chi(c',d)(\gamma)=\chi(c',d)(\gamma(c,d)\ast u) = \delta(c',d') \ast u^\ddagger$$
where $d'\in M\Delta(c')$ and $u^\ddagger\colon d' \to d$ is the transpose of $u$.
Now, define the idempotent cones $\gamma^\dagger=\gamma(c,d)$,  $\gamma^*=\gamma(c',d')$, $\delta^\dagger=\delta(c',d')$ and $\delta^*=\delta(c,d)$. Then as in Lemma \ref{lemab} and Proposition \ref{prog}, we can verify that $\gamma^\dagger \ \mathrel{\mathscr{R}}^* \ \gamma \ \mathrel{\mathscr{L}}^* \ \gamma^*$ and $\delta^\dagger \ \mathrel{\mathscr{R}}^* \ \delta \ \mathrel{\mathscr{L}}^* \ \delta^*$. Hence we have  $(\gamma^\dagger,\delta^*) \ \mathrel{\mathscr{R}}^* \ (\gamma,\delta) \ \mathrel{\mathscr{L}}^* \ (\gamma^*,\delta^\dagger)$. Thus $ \mathbb{S}\Omega$ is abundant. 

Finally, as described in Lemma \ref{lemic} and Proposition \ref{prog}, there are connecting isomorphisms $\alpha\colon \langle \gamma^\dagger\rangle \to \langle \gamma^* \rangle $ and $\beta\colon \langle \delta^\dagger\rangle \to \langle \delta^* \rangle $ in the concordant semigroups $\widehat{\Gamma}$ and $\widehat{\Delta}$, respectively. Then we can easily verify that the map $\upsilon\colon(\iota,\kappa)\mapsto((\iota)\alpha,(\kappa)\beta^{-1})$ for each $(\iota,\kappa)\in  \langle (\gamma^\dagger,\delta^*) \rangle$ is a connecting isomorphism between $\langle (\gamma^\dagger,\delta^*) \rangle$ and $\langle (\gamma^*,\delta^\dagger) \rangle$. Thus the semigroup $\mathbb{S}\Omega$ is concordant.
\end{proof}

Now, given a cross-connection $\Omega=(\Gamma,\Delta;\mathcal{C},\mathcal{D})$ with the cross-connection semigroup $\mathbb{S}\Omega$, define a functor $F_\Omega\colon \mathcal{C} \to \mathbb{L}(\mathbb{S}\Omega)$ as follows. For an arbitrary $c\in v\mathcal{C}$ and a morphism $f\in \mathcal{C}(c,c')$,
\begin{align}\label{eqnfo}
\begin{split}
vF_\Omega(c)=&\mathbb{S}\Omega (\gamma(c,d),\delta(c,d)) \ \text{ for some }d \in M\Delta(c);\\
F_\Omega(f)=\rho(&(\gamma(c,d),\delta(c,d)),(\gamma(c,d)\ast f^\circ,\delta),(\gamma(c',d'),\delta(c',d')))
\end{split}
\end{align}
where $d'\in M\Delta(c')$ and $\delta$ is any consistent cone in $\widehat{\Delta}$ with apex $d'$ such that it is linked to the cone $\gamma(c,d)\ast f^\circ$. Also, define a functor $G_\Omega\colon \mathcal{D} \to \mathbb{R}(\mathbb{S}\Omega)$ as follows. For an arbitrary $d\in v\mathcal{D}$ and a morphism $g\in \mathcal{D}(d,d')$,
\begin{align}\label{eqngo}
\begin{split}
vG_\Omega(d)=&(\gamma(c,d),\delta(c,d)) \mathbb{S}\Omega  \ \text{ for some }c \in M\Gamma(d);\\
G_\Omega(g)=\lambda(&(\gamma(c,d),\delta(c,d)),(\gamma,\delta(c,d)\ast g^\circ),(\gamma(c',d'),\delta(c',d')))
\end{split}
\end{align}
where $c'\in M\Gamma(d')$ and $\gamma$ is any consistent cone in $\widehat{\Gamma}$ with apex $c'$ such that it is linked to the cone $\delta(c,d)\ast g^\circ$. Then, emulating the proof of \cite[Theorem IV.35]{cross}, we can show that $F_\Omega$ and $G_\Omega$ are consistent category isomorphisms. Hence we have the following theorem.
\begin{thm}
For a cross-connection $\Omega=(\Gamma,\Delta;\mathcal{C},\mathcal{D})$ with the cross-connection semigroup $\mathbb{S}\Omega$, the consistent categories $\mathbb{L}(\mathbb{S}\Omega)$ and $\mathbb{R}(\mathbb{S}\Omega)$ are isomorphic to the categories $\mathcal{C}$  and $\mathcal{D}$, respectively.
\end{thm}

\section{Category equivalence}\label{seccat}  

In Section \ref{secdual}, we have seen how a concordant semigroup gives rise to a cross-connection and in Section \ref{seccxn}, we have constructed the concordant semigroup which arises from an abstract cross-connection of consistent categories. In this section, we proceed to extend this correspondence to a category equivalence between the category $\mathbf{CS}$ of concordant semigroups and the category $\mathbf{CC}$ of cross-connections of consistent categories. For this end, first we introduce morphisms in the category $\mathbf{CC}$ of cross-connections.

\begin{dfn}
Let $\Omega=(\Gamma,\Delta;\mathcal{C},\mathcal{D})$ and $\Omega'=(\Gamma',\Delta';\mathcal{C}',\mathcal{D}')$ be two cross-connections with biordered sets $E_\Omega$ and $E_{\Omega'}$, respectively. A \emph{CC-morphism} $m\colon \Omega\to \Omega'$ is a pair $m=(F_m, G_m)$ of functors $F_m\colon\mathcal{C}\to \mathcal{C}'$ and $G_m\colon\mathcal{D}\to \mathcal{D}'$ which satisfies the following axioms:
\begin{enumerate}
	\item [(M1)] The functors $F_m$ and $G_m$ preserve inclusions and bimorphisms.
	\item [(M2)] If $(c,d)\in E_\Omega$, then $(F_m(c),G_m(d))\in E_{\Omega'}$ and $$F_m(\gamma(c,d)(c'))=\gamma(F_m(c),G_m(d))(F_m(c'))\ \text{ for all }c'\in v\mathcal{C}.$$ 
    \item [(M3)] If $f^\ddagger\colon d' \to d$ is the transpose of $f\colon c\to c'$, then $G_m(f^\ddagger)=(F_m(f))^\ddagger$. 
\end{enumerate}
\end{dfn}

Given a CC-morphism $m=(F_m, G_m)\colon \Omega\to\Omega'$, for an arbitrary element $(\gamma,\delta)\in \mathbb{S}\Omega$ such that $\gamma=\gamma(c',d)\ast u\in \Gamma(c,d)$ for a bimorphism $u$ and $\delta=\chi(c,d)(\gamma)=\delta(c,d')\ast u^\ddagger\in \Delta(c,d)$, define a mapping $\mathbb{S}m\colon \mathbb{S}\Omega\to \mathbb{S}{\Omega'}$ as follows:
\begin{equation}\label{eqnsm}
((\gamma,\delta))\mathbb{S}m=(\gamma(F_m(c'),G_m(d))\ast F_m(u),\delta(F_m(c),G_m(d'))\ast G_m(u^\ddagger)).
\end{equation}

\begin{thm}\label{thmgdhom}
$\mathbb{S}m\colon \mathbb{S}\Omega\to \mathbb{S}{\Omega'}$ as defined above is a good homomorphism.
\end{thm}
\begin{proof}
First, exactly as shown in the proof of \cite[Theorem V.11]{cross}, we can show that $\mathbb{S}m$ is a homomorphism of semigroups $\mathbb{S}\Omega$ and $\mathbb{S}{\Omega'}$ such that $\mathbb{S}m$ is injective [surjective] if and only if $m$ is injective [surjective]. Further, let $(\gamma,\delta) \in \mathbb{S}\Omega$ such that $\gamma=\gamma(c,d)\ast u$ for a bimorphism $u\colon c \to c'$ and $\delta=\chi(c',d)(\gamma)$. Then as in the proof of Theorem \ref{thmcxncon}, we can find  idempotents $(\gamma^\dagger,\delta^*), (\gamma^*,\delta^\dagger)\in \mathbb{S}\Omega$ such that $(\gamma^\dagger,\delta^*) \ \mathrel{\mathscr{R}}^* \ (\gamma,\delta) \ \mathrel{\mathscr{L}}^* \ (\gamma^*,\delta^\dagger)$. Then since $F_m$ and $G_m$ preserve bimorphisms, we have $F_m(u)$ and $G_m(u^\ddagger)$ are bimorphisms. Then we can easily verify that $((\gamma^\dagger,\delta^*) )\mathbb{S}m\ \mathrel{\mathscr{R}}^* \ ((\gamma,\delta))\mathbb{S}m \ \mathrel{\mathscr{L}}^* \ ((\gamma^*,\delta^\dagger))\mathbb{S}m$.
\end{proof}

\begin{thm}\label{funs}
Further, the assignments
\begin{equation*}
v\mathbb{S} \colon \Omega \mapsto  \mathbb{S}\Omega \quad \quad  \mathbb{S}\colon m\mapsto \mathbb{S}m
\end{equation*}
is a functor $\mathbb{S}\colon \mathbf{CC} \to \mathbf{CS}$ from the category $\mathbf{CC}$ of cross-connections of consistent categories to the category $\mathbf{CS}$ of concordant semigroups.
\end{thm}
The proof is a straightforward generalisation of the proof of \cite[Theorem V.13]{cross} and hence we omit it.

\begin{thm}\label{thmccmor}
Given a good homomorphism $h\colon S\to S'$ of concordant semigroups, define functors $F_h\colon \mathbb{L}(S)\to\mathbb{L}(S')$ and $G_h\colon \mathbb{R}(S)\to\mathbb{R}(S')$ as follows:
\begin{align}\label{eqnfhgh}
\begin{split}
vF_h(Se)=S'(eh), &\quad F_h(\rho(e,u,f)) = \rho (eh,uh,fh),\\
vG_h(eS)=(eh)S' \quad &\text{ and } \quad G_h(\lambda(e,u,f)) = \lambda (eh,uh,fh).
\end{split}
\end{align}
Then the pair of functors $\Omega h =(F_h,G_h)$ is a CC-morphism between the cross-connections $\Omega S = (\mathbb{L}(S),\mathbb{R}(S);\Gamma_S,\Delta_S)$ and $\Omega S'=(\mathbb{L}(S'),\mathbb{R}(S');\Gamma_{S'},\Delta_{S'})$.	
\end{thm}
The proof of \cite[Theorem V.14]{cross} gives the routine verification of the above theorem and further describes a  functor from the category $\mathbf{CS}$ of concordant semigroups to the category $\mathbf{CC}$ of cross-connections of consistent categories.

\begin{thm}\label{func}
The assignments
\begin{equation*}
v\mathbb{C} \colon S\mapsto \Omega S \quad \quad  \mathbb{C}\colon h\mapsto \Omega S
\end{equation*}
define a functor $\mathbb{C}\colon \mathbf{CS} \to \mathbf{CC}$.
\end{thm}

Thus we have built two functors $\mathbb{S}\colon \mathbf{CC} \to \mathbf{CS}$ and $\mathbb{C}\colon \mathbf{CS} \to \mathbf{CC}$ between the categories of cross-connections and concordant semigroups. Now, we proceed to prove an adjoint equivalence between the categories using these functors. 

For this end, we require the following proposition whose proof carries over to the more general class of  \emph{weakly $U$-abundant} (also called $U$-semiabundant) semigroups \cite{wangureg,lawsonordered0}. The proof is due to Victoria Gould (in a personal communication) and this may be helpful in the future generalisations of this article.

\begin{pro}\label{prowr}
An abundant semigroup is weakly reductive.
\end{pro}
\begin{proof}
Recall that a semigroup $S$ is weakly reductive if the map $a\mapsto(\rho_a,\lambda_a)$ is injective where $\rho_a$ is the right regular representation of $S$ as defined in Proposition \ref{prosr} and $\lambda_a$ is the dual left regular representation. Let $S$ be an abundant semigroup. Suppose $(\rho_a,\lambda_a)=(\rho_b,\lambda_b)$ for $a,b\in S$. Since $S$ is abundant, there exists idempotents $a^\dagger,a^* \in S$ such that $a^\dagger\: \mathrel{\mathscr{R}}^* \: a \: \mathrel{\mathscr{L}}^* \: a^*$. So,  
$$a=a^\dagger a= a^\dagger \rho_a = a^\dagger \rho_b = a^\dagger b \text{ and } a=aa^*=\lambda_a a^* = \lambda_b a^*= ba^*.$$
Hence $a=a^\dagger b = b a^*$. Similarly, $b = b^\dagger a = ab^*$ for idempotents $b^\dagger,b^* \in S$ such that $b^\dagger\: \mathrel{\mathscr{R}}^* \: b \: \mathrel{\mathscr{L}}^* \: b^*$. Then 
$$a=ba^*=(b^\dagger a) a^* = b^\dagger (aa^*)=   b^\dagger a = b.$$
Hence $S$ is weakly reductive.
\end{proof}
\begin{thm}\label{thmadjcs}
For each concordant semigroup $S$, define $\varphi(S)\colon S \to \mathbb{C}\mathbb{S}(S)$ as
$$\varphi(S)\colon a \mapsto (\rho^a,\lambda^a)$$
where $\rho^a$ and $\lambda^a$ are principal cones determined by $a$ in the categories $\mathbb{L}(S)$ and $\mathbb{R}(S)$ defined by (\ref{eqnprinc}) and its dual, respectively. Then $\varphi(S)$ is an isomorphism and the assignment $S\mapsto\varphi(S)$ is a natural isomorphism between the functors $1_{\mathbf{CS}}$ and $\mathbb{C}\mathbb{S}$. 
\end{thm}
\begin{proof}
First, observe that for a concordant semigroup $S$ with the cross-connection $\Omega S=(\mathbb{L}(S),\mathbb{R}(S);\Gamma_S,\Delta_S)$, as argued in \cite[Proposition IV.37]{cross}, any idempotent of the concordant semigroup $\widehat{\Gamma_S}$ is of the form $\rho^e=\gamma(Se,eS)$. Then using Lemma \ref{lemug}, we can see that any consistent cone in  $\widehat{\Gamma_S}$ is of the form $ \gamma(Se,eS)\ast \rho(e,a,f)$ where $\rho(e,a,f)$ is a bimorphism in $\mathbb{L}(S)$. Similarly, any consistent cone in  $\widehat{\Delta_S}$ is of the form $ \delta(Sf,fS)\ast \lambda(f,a,e)$ where $\lambda(f,a,e)$ is a bimorphism in $\mathbb{R}(S)$. Then as shown in \cite[Proposition IV.37]{cross}, we can see that the concordant semigroups $\widehat{\Gamma_S}$ and $\widehat{\Delta_S}$ defined by (\ref{eqnug}) are given by:
$$\widehat{\Gamma_S}=\{\rho^a: a\in S\}\text{ and }\widehat{\Delta_S}=\{ \lambda^a:a\in S \}.$$
Further, as in \cite[Theorem IV.38]{cross}, the cross-connection semigroup $\mathbb{S}\Omega S$ is given by:
$$\mathbb{S}\Omega S=\{ (\rho^a,\lambda^a) :a\in S \}.$$
This implies that $\varphi(S)$ is surjective. By Proposition \ref{prosr} and its dual, we see that the map $\varphi(S)$ is a homomorphism. By Proposition \ref{prowr}, a concordant semigroup is {weakly reductive}. So the last statement of Proposition \ref{prosr} and its dual imply that $a \mapsto (\rho^a,\lambda^a)$ is injective. Hence $\varphi(S)$ is an isomorphism.

Now, to show that $\varphi\colon S\mapsto\varphi(S)$ is a natural transformation, i.e., for a good homomorphism $h\colon S\to S'$ of concordant semigroups, we have to show that the following diagram commutes:
\begin{equation*}\label{}
\xymatrixcolsep{3pc}\xymatrixrowsep{4pc}\xymatrix
{
	S \ar[d]_{h}\ar[rr]^{\varphi(S)}   	&& \mathbb{S}\Omega S \ar[d]^{\mathbb{S}\Omega h} \\       
	S' \ar[rr]^{\varphi(S')} && \mathbb{S} \Omega S'
}
\end{equation*}
For $a\in S$, we have 
$$ah\varphi(S')=(\rho^{ah},\lambda^{ah}).$$
Also,
\begin{align*}\label{}
\begin{split}
a\varphi(S)\mathbb{S}\Omega h &= (\rho^a,\lambda^a)\mathbb{S}\Omega h\\
&=(\gamma(Se,eS)\ast\rho(e,a,f),\delta(Sf,fS)\ast \lambda(f,a,e))\mathbb{S}\Omega h\quad (\text{as discussed above})\\
&=(\gamma(S'eh,ehS')\ast\rho(eh,ah,fh),\delta(S'fh,fhS')\ast \lambda(fh,ah,eh)) \ (\text{by (\ref{eqnsm}) and (\ref{eqnfhgh})})\\
&=(\rho^{eh}\ast\rho(eh,ah,fh),\lambda^{fh}\ast \lambda(fh,ah,eh))\\
&=(\rho^{ehah},\lambda^{ahfh})\\
&=(\rho^{(ea)h},\lambda^{(af)h})\quad (\text{since $h$ is a good homomorphism})\\
&=(\rho^{ah},\lambda^{ah}).\\
\end{split}
\end{align*}
So, the above diagram commutes and hence $\varphi$ is a natural isomorphism.	
\end{proof}

\begin{thm}\label{thmadjsc}
For each cross-connection $\Omega=(\Gamma,\Delta;\mathcal{C},\mathcal{D})$, let
$$\psi(\Omega) = (F_\Omega,G_\Omega)$$
where $F_\Omega\colon \mathcal{C} \to \mathbb{L}(\mathbb{S}\Omega)$ and $G_\Omega\colon \mathcal{D} \to \mathbb{R}(\mathbb{S}\Omega)$ are isomorphisms as defined in (\ref{eqnfo}) and (\ref{eqngo}), respectively. Then $\psi(\Omega)$ is an isomorphism of cross-connections and the mapping 
$$\Omega\mapsto\psi(\Omega)$$
is a natural isomorphism $\psi\colon1_{\mathbf{CC}}\to \mathbb{S}\mathbb{C}$.
\end{thm}
We omit the proof as an exact adaptation of the proof of \cite[Theorem V.17]{cross} suffices.
\begin{thm}\label{cateq}
The category $\mathbf{CS}$ of concordant semigroups is equivalent to the category $\mathbf{CC}$ of cross-connections of consistent categories.
\end{thm}
\begin{proof}
By Theorem \ref{thmadjcs} and Theorem \ref{thmadjsc}, it is clear that $(\mathbb{C},\mathbb{S},\varphi,\psi)\colon \mathbf{CS}\to \mathbf{CC}$ is an adjoint equivalence. Hence the theorem.
\end{proof}

\section{Consistent categories, normal categories and inductive cancellative categories}  

Recall from Lemma \ref{lemugrb} that if we specialise our discussion in Section \ref{seccons}-\ref{seccxn} to normal categories, we obtain a cross-connection $\overline{\Omega}=(\overline{\Gamma},\overline{\Delta};\overline{\mathcal{C}},\overline{\mathcal{D}})$ of normal categories $\overline{\mathcal{C}}$ and $\overline{\mathcal{D}}$. Further, extending the discussion, by Theorem \ref{thmcxncon} we can obtain a regular cross-connection semigroup $S\overline{\Omega}$ and using Theorem \ref{cateq} we have the following result of \cite[Theorem V.18]{cross}:
\begin{thm}
The category $\mathbf{RS}$ of regular semigroups is equivalent to the category $\mathbf{Cr}$ of cross-connections of normal categories.
\end{thm}

Now, we proceed to describe the relationship between our approach and Armstrong's approach using inductive cancellative categories \cite{armstrong}. We refer the reader to \cite{mem} for the formal definitions of regular biordered set, $E$-paths, singular $E$-squares etc. We begin by recalling the definition of an ordered cancellative category.

\begin{dfn}\label{dfnocc}
Let $\mathcal{I}$ be a small category and $\leq$ a partial order on $\mathcal{I}$.  Let $e,f \in v\mathcal{I}$ and $x,y$ etc denote arbitrary morphisms of $\mathcal{I}$ such that $\mathbf{d}(x)$ and $\mathbf{r}(x)$ is the domain and codomain, respectively of an arbitrary morphism $x$. Then $(\mathcal{I},\leq)$ is called an \emph{ordered cancellative category} if the following hold.
\begin{enumerate}
		\item [(OCC1)] Every morphism in $\mathcal{I}$ is a bimorphism.
		\item [(OCC2)] If $u\leq x$, $v\leq y$ and $\mathbf{r}(u)=\mathbf{d}(v)$, $\mathbf{r}(x)=\mathbf{d}(y)$, then $uv \leq xy$.
		\item [(OCC3)] If $x\leq y$, then $1_{\mathbf{d}(x)}\leq 1_{\mathbf{d}(y)}$ and $1_{\mathbf{r}(x)}\leq 1_{\mathbf{r}(y)}$.
		\item [(OCC4)] If $1_e\leq 1_{\mathbf{d}(x)}$, then there exists a unique element $e{\downharpoonleft} x$ (called the \emph{restriction} of $x$ to $e$) in $\mathcal{I}$ such that $e{\downharpoonleft} x\leq x$ and $\mathbf{d}(e{\downharpoonleft} x) = e$.
		\item [(OCC5)] If $1_f\leq 1_{\mathbf{r}(x)}$, then there exists a unique element $x{\downharpoonright} f$ (called the \emph{corestriction} of $x$ to $f$) in $\mathcal{I}$ such that $x{\downharpoonright} f \leq x$ and $\mathbf{r}(x{\downharpoonright} f) = f$.
		\end{enumerate}
\end{dfn} 

\begin{dfn}\label{dfnicc}
Let $(\mathcal{I},\leq)$ be an {ordered cancellative category} with $v\mathcal{I}=E$ a regular biordered set such that $\mathrel{\omega}$ coincides with $\leq$ on $E$. Suppose for $e,f \in E$ satisfying $e\mathrel{\mathscr{R}} f$ and $e\mathrel{\mathscr{L}}f$, there is a distinguished morphism $[e,f]$ from $e$ to $f$ such that 
\begin{enumerate}[(i)]
	\item $[e,e] = 1_e$;
	\item if $e\mathrel{\mathscr{R}}f\mathrel{\mathscr{R}}g$ or $e\mathrel{\mathscr{L}}f\mathrel{\mathscr{L}}g$  then $[e,f][f,g]=[e,g]$;
	\item if $[g,h]$ exists and $e\mathrel{\omega} g$ then $[e,f]$ exists with $f=heh$ and $[e,f] \leq [g,h]$.
\end{enumerate}
Then $(\mathcal{I},\leq )$ is an \emph{inductive cancellative category} if the following axioms and their duals hold.
\begin{enumerate}
\item[(ICC1)] Let $x\in \mathcal{I}$ and for $i=1,2$, let $e_i$, $f_i \in E$ such that $e_i \leq \mathbf{d}(x)$ and $f_i = \mathbf{r}(1_{e_i}{\downharpoonleft} x)$. If $e_1\mathrel{\omega}^r e_2$, then $f_1\mathrel{\omega}^r f_2$, and
$$[e_1,e_1e_2](e_1e_2{\downharpoonleft} x) = (e_1{\downharpoonleft} x)[f_1,f_1f_2].$$
\item[(ICC2)] If $\bigl[ \begin{smallmatrix} e&f\\ g&h \end{smallmatrix} \bigr]$ is a singular E-square, then $[e,f][f,h] = [e,g][g,h]$.
\end{enumerate}
\end{dfn}

Let $\Omega=(\mathcal{C},\mathcal{D};{\Gamma},\Delta)$ be a cross-connection of consistent categories. We proceed to identify the inductive cancellative category $\mathcal{I}(\Omega)$ associated with the cross-connection $\Omega$.

Clearly, $v\mathcal{I}(\Omega)$ is the regular biordered set $E_\Omega$ as described in Lemma \ref{lemesg}. In the sequel, as in the lemma, we shall identify the idempotent cone $(\gamma(c,d),\delta(c,d))$ (and hence the identity morphisms of the category $\mathcal{I}(\Omega)$) with the pair of objects $(c,d)\in v\mathcal{C}\times v\mathcal{D}$. Hence,
$$E_\Omega=\{ (c,d)\in v\mathcal{C}\times v\mathcal{D} : c\in M\Gamma(d)\}.$$

Given two objects $(c,d),(c',d')\in v\mathcal{I}(\Omega)$, any bimorphism $u\colon c \to c'$ in the category $\mathcal{C}$ is defined as a morphism in the category $\mathcal{I}(\Omega)$ from $(c,d)$ to $(c',d')$.
Then, corresponding to the bimorphism $u$, as in Theorem \ref{thmcxncon}, there is a connecting isomorphism $\upsilon\colon \langle (c,d) \rangle \to \langle (c',d') \rangle$. 

Given any two morphisms $u_1\colon(c_1,d_1)\to(c_1',d_1')$ and $u_2\colon(c_2,d_2)\to(c_2',d_2')$ in the category $\mathcal{I}(\Omega)$ with connecting isomorphisms $\upsilon_1$ and $\upsilon_2$, respectively, we define a relation $\leq_\Omega$ as follows:
\begin{equation*}\label{eqnleqo}
u_1\leq_\Omega u_2 \iff 
(c_1,d_1) \subseteq (c_2,d_2), \ u_1= (j(c_1,c_2)u_2)^\circ  \text{ and } (c_1',d_1')=(c_1,d_1)\upsilon_2
\end{equation*}
where $(j(c_1,c_2)u_2)^\circ$ is the epimorphic component of the monomorphism $j(c_1,c_2)u_2$. It can be easily verified that $\leq_\Omega$ is a partial order on $\mathcal{I}(\Omega)$. 

Further, given a morphism $u\colon (c,d)\to(c',d')$ in $\mathcal{I}(\Omega)$ such that $(c_1,d_1)\subseteq (c,d)$, then we define the restriction $(c_1,d_1){\downharpoonleft}u$ (of the morphism $u$ to $(c_1,d_1)$) as the morphism $(j(c_1,c)u)^\circ$. Similarly, for the morphism $u\colon (c,d)\to(c',d')$ in $\mathcal{I}(\Omega)$ with connecting isomorphism $\upsilon$ such that $(c_1',d_1')\subseteq (c',d')$, we define $(c_1,d_1)=(c_1',d_1')\upsilon^{-1}$. Then, the corestriction $u{\downharpoonright}(c_1',d_1')$ (of the morphism $u$ to $(c_1',d_1')$) is defined as the morphism $(j(c_1,c)u)^\circ$. Then we can easily verify that $(\mathcal{I}(\Omega),\leq_\Omega)$ is an ordered cancellative category. 

Finally, for $(c,d),(c',d')\in v\mathcal{I}(\Omega)$ such that $(c,d)\mathrel{\mathscr{R}}(c',d')$ or $(c,d)\mathrel{\mathscr{L}}(c',d')$, 
we define the isomorphism $\gamma(c',d')(c)$ as the distinguished morphism in $\mathcal{I}(\Omega)$ from $(c,d)$ to $(c',d')$. Hence, we can verify the following theorem:
\begin{thm}
$(\mathcal{I}(\Omega),\leq_\Omega)$ is an inductive cancellative category.
\end{thm}
Further, given a CC-morphism $m\colon \Omega\to \Omega'$ between cross-connections $\Omega=(\Gamma,\Delta;\mathcal{C},\mathcal{D})$ and $\Omega'=(\Gamma',\Delta';\mathcal{C}',\mathcal{D}')$, we can easily verify that $m_{|\mathcal{I}(\Omega)}\colon\mathcal{I}(\Omega)\to \mathcal{I}(\Omega')$ is an inductive functor in the sense of \cite{armstrong}. 
Thus we obtain a functor $\mathbb{I}$ from the category $\mathbf{CC}$ of cross-connections to the category $\mathbf{ICC}$ of inductive cancellative categories.

Further, generalising the discussion in \cite[Section IV]{indcxn2}, we can construct an adjoint inverse functor $\mathbb{I}'\colon\mathbf{ICC}\to \mathbf{CC}$. Using these functors, we can prove the following direct equivalence, whose proof we omit.
\begin{thm}
The category $\mathbf{CC}$ of cross-connections of consistent categories is equivalent to the category $\mathbf{ICC}$ of inductive cancellative categories.
\end{thm}

\appendix
\section{Cross-connection structure of regular semigroups}\label{appcxnrs}

As mentioned in Section \ref{sec1}, in \cite{cross}, a regular semigroup was constructed from a pair of cross-connected normal categories. The construction is as follows: given an abstractly defined normal category $\mathcal{C}$, we first associate with it an intermediary regular semigroup called the semigroup $\widehat{\mathcal{C}}$ of normal cones. 

It can be seen that given a regular semigroup $S$, its principal left ideals with partial right translations as morphisms and principal right ideals with partial left translations as morphisms, form normal categories $\mathbb{L}(S)$ and $\mathbb{R}(S)$, respectively. Then their corresponding semigroups of normal cones, namely $\widehat{\mathbb{L}(S)}$ and $\widehat{\mathbb{R}(S)}$, will give representations of the regular semigroup we started with.

The interrelationship of the categories $\mathbb{L}(S)$ and $\mathbb{R}(S)$ is abstracted using the notion of a cross-connection. Via the cross-connection, certain normal cones of the  semigroup $\widehat{\mathbb{L}(S)}$ can be `linked' with those of the semigroup $\widehat{\mathbb{R}(S)}$. The collection of all such linked normal cones will form a regular semigroup called the cross-connection semigroup. 

Thus, starting with a pair of abstractly defined cross-connected normal categories $\mathcal{C}$ and $\mathcal{D}$, we can construct a regular cross-connection semigroup as a subdirect product of the regular semigroups $\widehat{\mathcal{C}}$ and $\widehat{\mathcal{D}}$. Conversely, given any regular semigroup, we obtain a pair of cross-connected normal categories: namely $\mathbb{L}(S)$ and $\mathbb{R}(S)$. This correspondence is shown to be a category equivalence.

\section*{Acknowledgements}
\noindent The authors express their heartfelt gratitude to Prof.~L\'{a}szl\'{o}~M\'{a}rki for his keen interest in this article which led to its realisation after several years of deferment.
We thank Prof.~Mikhail~V.~Volkov for his guidelines and suggestions which have helped improve the manuscript considerably.
We are also grateful to a referee for extensive comments on an early version of the manuscript.
The first author thanks Prof.~Victoria~Gould for several fruitful discussions (and in particular for the proof of Proposition \ref{prowr}) during the author's visit to the University of York in January 2018.\\
This is an extended and revised version of the Ph.D. thesis \cite{romeo}\nocite{romeo1} of the second author. The major differences in the construction are that the balanced, reductive categories of \cite{romeo} are called consistent categories here and that we define cross-connections using two functors, instead of a single functor.\\ 	

\bibliographystyle{plain}

\end{document}